\documentclass{amsart}
\usepackage[latin1]{inputenc}
\usepackage[english]{babel}
\usepackage[T1]{fontenc}
\usepackage{amsmath}
\usepackage{amssymb}
\usepackage{amsthm}
\usepackage{array,booktabs}
\usepackage{graphicx}
\usepackage{lmodern}
\usepackage{tcolorbox}
\usepackage{quoting}
\usepackage{enumerate}
\usepackage{mathtools}
\usepackage{csquotes}
\usepackage{cite}
\usepackage{tikz}
\usepackage{tikz-cd}
\usetikzlibrary{matrix,arrows,decorations.pathmorphing}
\usepackage{wrapfig}
\renewcommand{\Re}{\operatorname{Re}}
\renewcommand{\Im}{\operatorname{Im}}
\def\R{\ensuremath\mathbb{R}}
\def\C{\ensuremath\mathbb{C}}
\def\A{\ensuremath\mathbb{A}}
\def\Z{\ensuremath\mathbb{Z}}
\def\Q{\ensuremath\mathbb{Q}}

\def\H{\ensuremath\mathbb{H}}

\usepackage{hyperref}
\usepackage[final]{pdfpages}
\usepackage{verbatim}
\newtheorem{thm}{Theorem}[section]

\newtheorem{cor}[thm]{Corollary}
\newtheorem{lemma}[thm]{Lemma}
\newtheorem{prop}[thm]{Proposition}
\theoremstyle{remark}
\newtheorem{remark}[thm]{Remark}

\def\eps{\ensuremath\varepsilon}
\def\0{\emptyset}

\def\Id {\text{\rm Id}}
\def\dist{\text{\rm dist}}
\def\spec{\text{\rm spec}}

\def\SL{\hbox{\rm SL}}
\def\PSL{\hbox{\rm PSL}}
\def\GL{\hbox{\rm GL}}
\def\modulo{\text{ \rm mod }}
\def\vol{\text{\rm vol}}
\def\cond{\text{\rm cond}}
\numberwithin{equation}{section}
\def\n{A}
\def\nn{a}
\def\mm{b}
\def\m{B}

\numberwithin{equation}{section}
\begin{document}
\title{Central values of additive twists of cuspidal $L$-functions}
\author{Asbjørn Christian Nordentoft}
\address{Mathematical Institute of the University of Bonn, Endenicher Allee 60, Bonn 53115, Germany}
\email{\href{mailto:acnordentoft@outlook.com}{acnordentoft@outlook.com}}
\date{\today}
\begin{abstract} Additive twists are important invariants associated to holomorphic cusp forms; they encode the Eichler--Shimura isomorphism and contain information about automorphic $L$-functions. In this paper we prove that central values of additive twists of the $L$-function associated to a holomorphic cusp form $f$ of even weight $k$ are asymptotically normally distributed. This generalizes (to $k\geq 4$) a recent breakthrough of Petridis and Risager concerning the arithmetic distribution of modular symbols. Furthermore we give as an application an asymptotic formula for the averages of certain \lq wide{\rq} families of automorphic $L$-functions consisting of central values of the form $L(f\otimes \chi,1/2)$ with $\chi$ a Dirichlet character. 
\end{abstract}
\maketitle
\section{Introduction} 
In this paper we study the statistics of central values of additive twists of the $L$-functions of holomorphic cusp forms (of arbitrary even weight). Additive twists of cuspidal $L$-functions are important invariants; on the one hand they show up in the parametrization of the Eichler--Shimura isomorphism and on the other hand additive twists shed light on central values of Dirichlet twists of cuspidal $L$-functions.

We prove that when arithmetically ordered, the central values of the additive twists of a cuspidal $L$-function are asymptotically normally distributed. As an application we calculate the asymptotic behavior (as $X\rightarrow \infty$) of the averages of certain \lq wide\rq\, families of automorphic $L$-functions;
\begin{align}\label{family}  \sideset{}{^\ast}\sum_{\pi \in\mathcal{F}_n,\, \cond(\pi)\leq X} L(\pi, 1/2),   \end{align}
where $\cond(\pi)$ denotes the conductor of the automorphic representation $\pi$, the asterisk on the sum denotes a certain weighting and the families consist of isobaric sums of twists; 
$$\mathcal{F}_n=\left\{ (\pi_f\otimes \chi_1) \boxplus \cdots \boxplus (\pi_f\otimes \chi_{2n})\mid \chi_1\cdots \chi_{2n}={\bf 1} \right\},$$
where $\chi_1,\ldots, \chi_{2n}$ are automorphic representations of $\GL_1(\A_\Q)$, ${\bf 1}$ denotes the trivial automorphic representation of $\GL_1(\A_\Q)$ and $\pi_f$ is the automorphic representation of $\GL_2(\A_\Q)$ associated to a {\it fixed} holomorphic newform $f$ (suppressed in the notation). For the precise statements of our main results, see Theorem \ref{mainthm} and Corollary \ref{multiplicativetwistsgeneral} below. 

Our distribution result is a higher weight analogue of a recent result of Petridis and Risager \cite{PeRi}, which settled (an averaged version of) a conjecture due to Mazur and Rubin \cite{MaRu19} concerning the normal distribution of modular symbols (a different proof was later given by Lee and Sun \cite{LeeSun19} using dynamical methods).
The conjecture of Mazur and Rubin concerns the arithmetic distribution of the modular symbol map;
$$  \{\infty, \mathfrak{a}\} \mapsto \langle \mathfrak{a}, f \rangle:=2\pi i \int_{\infty}^\mathfrak{a} f(z)dz,  $$
where $f\in \mathcal{S}_2(\Gamma_0(q))$ is a cusp form of weight $2$ and level $q$ and $ \{\infty, \mathfrak{a}\}$ is the homology class of curves between the cusps $\infty$ and $\mathfrak{a}$. Petridis and Risager prove that this map is asymptotically normally distributed when ordered by the denominator of the cusp $\mathfrak{a}$ and appropriately normalized \cite[Theorem 1.10]{PeRi}. See Section \ref{mazurrubinback} below for more background on the conjectures of Mazur and Rubin and their motivation. 
\subsection{Statement of results} 
We will now state a special case of our main result and refer to Theorem \ref{maingeneral} below for the most general version. Let $f\in S_k(\Gamma_0(q))$ be a cusp form of (arbitrary) even weight $k$ and level $q$ with Fourier expansion (at $\infty$) given by
$$f(z)=\sum_{n\geq 1} a_f(n) q^n,\quad q=e^{2\pi i z}.$$
Then we define the {\it additive twist} (by $r\in \R$) of the $L$-function associated to $f$ as
$$ L(f,r,s):= \sum_{n\geq 1} \frac{a_f(n) e(nr)}{n^s}, $$
where $e(x)=e^{2\pi i x}$. The additive twists converge absolutely for $\Re s>\frac{k+1}{2}$ and when $r\in \Q$, they admit analytic continuation to the entire complex plane. If furthermore $r$ is $\Gamma_0(q)$ equivalent to $\infty$, we have simple functional equations relating $s\leftrightarrow k-s$ (see Section \ref{Additive} below for details). For $k=2$ the additive twists coincide with modular symbols (see Remark \ref{conrete}).  

Now we will explain what we mean by saying that additive twists are {\it asymptotically normally distributed}: Given $X>0$, we consider $L(f,\cdot,k/2)$ as a random variable defined on the following outcome space endowed with the uniform probability measure;
\begin{align}\label{TQ}  T(X):= \left\{  a/c \in \Q \mid 0<a<c\leq X,\, (a,c)=1,\, q\mid c  \right\}.\end{align}
Then our main theorem is the following.
\begin{thm} \label{mainthm}
Let $f\in \mathcal{S}_k(\Gamma_0(q))$ be a cusp form of even weight $k$ and level $q$. Then for any fixed box $\Omega\subset \C$, we have 
\begin{align}
\mathbb{P}_{T(X)}\left( \frac{L(f,a/c,k/2)}{(C_f \log c)^{1/2}} \in \Omega \right):=& \frac{\#\{ a/c\in T(X)\mid \frac{L(f,a/c,k/2)}{(C_f \log c)^{1/2}} \in \Omega  \}}{\#T(X)}\\
\nonumber =&\mathbb{P}( \mathcal{N}_\C (0,1)\in \Omega)+o(1),
\end{align}
as $X\rightarrow \infty$, where $\mathcal{N}_\C (0,1)$ denotes the standard complex normal distribution and the variance is given by
\begin{align}\label{Cf}C_f:=\frac{(4\pi)^{k}|\!|f|\!|^{2}}{(k-1)!\, \vol(\Gamma_0(q))},\end{align}
with $|\!|f|\!|$ the Petersson-norm of $f$ and $\vol(\Gamma_0(q))$ the hyperbolic volume of $\Gamma_0(q)\backslash \H$.\\
(Here $\mathbb{P}( \mathcal{N}_\C (0,1)\in \Omega)$ denotes the probability of the event $ \mathcal{N}_\C (0,1)\in \Omega$.)

 \end{thm} 
 \begin{remark}\label{remgen}
With our methods, we can generalize the above theorem in three aspects, which can all be combined:
\begin{enumerate}
\item We can consider more general outcome spaces corresponding to twists at an arbitrary cusp. Note that $T(X)$ corresponds exactly to additive twists at cusps, which are $\Gamma_0(q)$-equivalent to $\infty$. See (\ref{T}) below for the definition of the outcome space for general cusps.
\item We can consider cusp forms for a general discrete and co-finite subgroup of $\PSL_2(\R)$ with a cusp at $\infty$.
\item We can consider the joint distribution of an orthogonal basis of cusp forms.  
\end{enumerate}
See Theorem \ref{maingeneral} below for the most general version of our main theorem, incorporating all of the three above aspects. 

 \end{remark}
\begin{remark}The constant $C_f$ is a higher weight analogue of the {\it variance slope} defined by Mazur and Rubin (see \cite[Theorem 1.9]{PeRi}). Note that $C_f$ is independent of the embedding $f\hookrightarrow \mathcal{S}_k(\Gamma_0(N))$. \end{remark}
\begin{remark}Independently, a different proof of Theorem \ref{mainthm} in the special case of trivial level $q=1$ was obtained by Bettin and Drappeau \cite{BeDr19} using dynamical methods similar to those used by Sun and Lee. It is still an open problem to extend the dynamical approach to deal with general level, but in return the dynamical approach of Bettin and Drappeau applies to much more general {\it quantum modular forms} in the sense of Zagier \cite{Zagier10}. It is unclear whether the automorphic methods of this paper can be generalized to deal with quantum modular forms as well. The similarities and differences between the automorphic and dynamical approach deserve further exploration.
\end{remark}
\begin{remark}\label{conrete} In more concrete terms the above theorem says that for any fixed real numbers $x_1<x_2$ and $y_1<y_2$, we have
\begin{align*}
\frac{\#\left\{\frac{a}{c}\in T(X)\mid  x_1\leq \Re \left( \frac{L(f,a/c,k/2)}{(C_f \log c)^{1/2}}\right)\leq x_2,y_1\leq \Im \left( \frac{L(f,a/c,k/2)}{(C_f \log c)^{1/2}}\right)\leq y_2\right\}}{\#T(X)}\\
 \rightarrow \frac{1}{2\pi}\int_{x_1}^{x_2} \int_{y_1}^{y_2} e^{-(x^2+y^2)/2} dx dy,
\end{align*}
as $X\rightarrow \infty$, which is exactly the formulation used in \cite{PeRi}. One can see that restricting to the case $k=2$ and letting $f\in \mathcal{S}_2(\Gamma_0(q))$, Theorem \ref{mainthm} above recovers \cite[Theorem 1.10]{PeRi}. Here one has to use the integral representation (\ref{periodint}) for the additive twist of the $L$-function, which shows $ \langle \mathfrak{a} , f\rangle= L(f,r_\mathfrak{a},1),$ where $r_\mathfrak{a}\in \R$ represents the cusp $\mathfrak{a}$ (i.e. $r_\mathfrak{a}$ is fixed by the parabolic subgroup $\Gamma_\mathfrak{a}$).
\end{remark}
\subsection{Moment calculations}
The proof of Theorem \ref{mainthm} uses the method of moments. The calculation of the moments of additive twists is of independent interest and is used in the application to automorphic $L$-functions in Corollary \ref{multiplicativetwistsgeneral} below. In the course of the paper we will evaluate a number of different moments. In particular we will prove the following result at the end of Section \ref{difcusp}, which is exactly what we need to conclude Corollary \ref{multiplicativetwistsgeneral}. 
\begin{thm} \label{normal}
Let $f\in S_k(\Gamma_0(q))$ be a cusp form of even weight $k$ and level $q$ and $n$ a non-negative integer. Then we have 
\begin{align}\label{momenter}\sum_{\substack{0<a<c\leq X\\ (qa,c)=1}} |L(f,a/c,k/2)|^{2n} =P_n(\log X)X^2+O_\eps(X^{ 4/3+\eps}),  \end{align}
where $P_n$ is a polynomial of degree $n$ with leading coefficient
$$\frac{q (2C_{f})^n\, n!}{\pi \, \vol(\Gamma_0(q))},$$
with $C_f$ as in (\ref{Cf}) above.
\end{thm}
\begin{remark}
The above moments correspond to additive twists at cusps which are $\Gamma_0(q)$-equivalent to the cusp $0$ (the set of all such twists is denoted $T_{\infty 0}$ in (\ref{T}) below). 
\end{remark}
\begin{remark} The determination of the moments follows from the analytic properties of a certain Eisenstein series $E^{m,n}(z,s)$ generalizing series introduced by Goldfeld in \cite{Go99} and \cite{Go99.2}. Determining the location of the dominating pole, the corresponding pole order and leading Laurent coefficient of the original Goldfeld Eisenstein series was firstly achieved by Petridis and Risager in \cite{PeRi2} using perturbation theory and the analytic properties of the hyperbolic resolvent. This allowed them to prove normal distribution for a certain more geometrically flavored ordering of the modular symbols (ordered by $c^2+d^2$, where $c,d$ are the lower entries of the matrix $\gamma$). In order to prove (an averaged version of) the conjecture of Mazur and Rubin, Petridis and Risager \cite{PeRi} essentially had to derive the analytic properties of the constant Fourier coefficient of $E^{m,n}(z,s)$. This is reminiscent of the Shahidi--Langlands method \cite[Section 8]{GeMi04}. The strategy of proof in this paper is inspired by the overall strategy introduced by Petridis and Risager. 
\end{remark}

\subsection{Applications to automorphic $L$-functions}\label{mazurrubinback} The motivation behind the conjectures of Mazur and Rubin was to gain information about the vanishing/non-vanishing of the central values of the twisted Hasse--Weil $L$-functions $L(E, \chi, 1)$ where $E/\Q$ is an elliptic curve and $\chi$ is a Dirichlet character. By a sufficiently general version of the Birch--Swinnerton-Dyer conjecture, this is related to the following problem in Diophantine stability (see \cite{MaRu19} for details):
\begin{align*} &\text{\it How likely is it that $\mathrm{rank}_\Z \, E(K)>\mathrm{rank}_\Z \,E(\Q)$}\\
&\qquad \qquad \qquad \qquad \text{\it as $K$ ranges over abelian extensions of $\Q$?}\end{align*}
If $\chi$ is a primitive Dirichlet character modulo $c$, then the Birch--Stevens formula \cite[Theorem 2.3]{MaRu19} relates these central values to modular symbols; 
$$ \tau(\overline{\chi})L(E, \chi, 1)= \sum_{a\in (\Z/c\Z)^\times} \overline{\chi}(a)  \langle a/c , f_E\rangle,  $$ 
where $f_E$ is the weight 2 newform corresponding to $E$ via modularity and $\tau(\overline{\chi})$ is a Gauss sum. This led Mazur and Rubin to the study of the distribution of modular symbols, and based on computational experiments, they made a number of conjectures about the distribution of modular symbols, one of which predicted a normal distribution. In this paper we will not contribute to the other conjectures put forth by Mazur and Rubin, consult instead \cite{DiHoKiLe20}, \cite{Sun20}, \cite[Theorem 9.2]{BlFoKoMiMiSa18}.

Following these lines of thinking, we apply our methods to the study of families constructed from certain twisted $L$-functions. Given a newform $f\in \mathcal{S}_k(\Gamma_0(q))$ of weight $k$ and level $q$ and a primitive Dirichlet character $\chi$ of conductor co-prime to $q$, we define the {\it multiplicatively twisted $L$-function of} $f$;
$$L(f\otimes \chi, s):= \sum_{n\geq 1} \frac{\lambda_f(n)\chi(n)}{n^s}, $$
(where $\lambda_f(n)$ denotes the $n$th Hecke eigenvalue of $f$), which converges absolutely for $\Re s>1$ and admits analytic continuation satisfying a functional equation. Note that because of the co-primality condition the above Dirichlet series defines (the finite part of) the $L$-function corresponding to the automorphic representation $\pi_f\otimes\chi$ (justifying the notation). 

The study of averages of multiplicative twists of cuspidal $L$-functions has a long history (see for instance the work of Rohrlich \cite{Rohrlich89}, Duke, Friedlander and Iwaniec \cite{DuFrIw01}, and Chinta \cite{Chinta02}). Recently Blomer, Fouvry, Kowalski, Michel, Mili{\'c}evi{\'c} and Sawin \cite{BlFoKoMiMiSa18} have given an extensive account of the second moment theory for such twists. We are able to obtain new results for these automorphic $L$-functions. Combining Theorem \ref{normal} and the (generalized) Birch--Stevens formula (see Lemma \ref{BirchSteven} below), we obtain asymptotic formulas for the following (arithmetically weighted) averages of certain \lq wide{\rq} families of multiplicatively twisted $L$-functions, making (\ref{family}) precise. 
 \begin{cor} \label{multiplicativetwistsgeneral} 
Let $f\in \mathcal{S}_k(\Gamma_0(q))$ be a newform of even weight $k$ and level $q$ and $n$ a non-negative integer. Then we have for all $\eps>0$
\begin{align}
\nonumber &\sum_{\substack{0<c\leq X, \\(c,q)=1}} \frac{1}{\varphi(c)^{2n-1}} \sum_{\substack{\chi_1,\ldots, \chi_{2n} \modulo c,\\ \chi_1\cdots \chi_{2n}= \chi_{\text{\rm principal}} }} \eps_{\chi_1,\ldots, \chi_n}\,\prod_{i=1}^{2n} \nu(f, \chi_i^*, c/c(\chi_i))L(f\otimes \chi_i^*,1/2)\\
\label{2ndmoment}&= P_n(\log X) X^2+O_\eps(X^{4/3+\eps}),
\end{align}
where $\chi^*\modulo c(\chi)$ denotes the primitive character inducing $\chi$, $\chi_{\text{\rm principal}}$ is the principal character mod $c$, $P_n$ is the degree $n$ polynomial from (\ref{momenter}), $\eps_{\chi_1,\ldots, \chi_n}=\chi_1(-1)\cdots \chi_n(-1)$ is a sign, and $\nu$ is an arithmetic weight given by
\begin{align}\label{arithmeticweight}\nu(f, \chi, n):= \tau(\overline{\chi})\sum_{\substack{n_1n_2n_3=n,\\ (n_1,q)=1}}\chi(n_1) \mu(n_1)\overline{\chi}(n_2) \mu(n_2) \lambda_f(n_3)n_3^{1/2}.\end{align}
\end{cor} 
In particular for $n=1$, the above corollary reduces to an average second moment (see Corollary \ref{average2ndmoment} below), which was calculated without the extra averaging in \cite[Theorem 1.17]{BlFoKoMiMiSa18}. Interestingly our methods completely avoid the use of approximate functional equations.
\begin{remark}
Observe that the sum in the arithmetic weight $\nu$ can be expressed as the triple convolution 
$$(\delta_{(\cdot, q)=1}\times \chi\times  \mu)\ast (\overline{\chi}\times \mu) \ast(\lambda_f \times  (\cdot)^{1/2}).$$ 
This was exploited by Bruggeman and Diamantis in \cite{BrDia16} in their study of Fourier coefficients of the Goldfeld Eisenstein series $E^{1,0}(z,s)$. Furthermore for $\chi$ primitive, we have 
$$\nu(f, \chi^*, c/c(\chi))=\nu(f, \chi, 1)=\tau(\overline{\chi}),$$ 
and in general, we have the bound $|\nu(f, \chi, n)|\ll (c(\chi) n)^{1/2}$ using Deligne's bound for the Hecke eigenvalues. 
\end{remark}
\begin{remark} Bettin \cite{Be17} has considered the Eisenstein case of the above theorems, which amounts to studying the Estermann function defined as 
$$  D(a/c,s):=\sum_{n\geq 1} \frac{d(n)e(na/c)}{n^s},$$
where $d(n)$ is the divisor function and $a/c\in \Q$. Bettin managed to calculate all moments averaging over $a\in (\Z\backslash c\Z)^\times$ using an approximate functional equation. He similarly applied his results to studying certain iterated moments of central values of Dirichlet $L$-functions. 
\end{remark}

\subsection*{Acknowledgements}
I would like to express my gratitude to my advisor Morten Risager for suggesting this problem to me and for our countless stimulating discussions. I would also like to thank Yiannis Petridis for his time and insight.  


\section{Method of proof}\label{method} 
Let $f\in \mathcal{S}_k(\Gamma_0(q))$ be a cusp form of even weight $k$ and level $q$. In this section we will describe the overall strategy of the proof of Theorem \ref{mainthm}. Our approach is an extension of the techniques developed by Petridis and Risager in \cite{PeRi2} and \cite{PeRi}. We would like to point out that many technical difficulties show up when $k\geq 4$ and some essential new ideas were needed. This includes using the lowering and raising operators in the analysis of the recursion formula (see for instance the proof of Lemma \ref{key}), the automorphic completion step as described in Section \ref{formulaforcentralvalue} and the use of what we call {\it $N$-shifted Golfeld Eisenstein series} defined in (\ref{N-shift}).

\subsection{The strategy of proof} We will use a classical result of Fréchet and Shohat \cite[p.17]{Serf} known as the {\it method of moments}; in order to get the sought-after convergence in distribution, it is enough to show that (after a suitable normalization) the asymptotic moments (as $X\rightarrow \infty$) of the central values;
$$   \sum_{\substack{a/c\, \in T(X)}} L(f,a/c,k/2)^m \overline{L(f,a/c,k/2)}^n,  $$
agree with those of the standard complex normal distribution. 

By a standard contour integration argument, this can be reduces to understanding the analytic properties of the following Dirichlet series
\begin{align} \label{lfunc} D^{m,n}(f, s):= \sum_{\gamma \in \Gamma_\infty\backslash \Gamma_0(q)/\Gamma_\infty} \frac{L(f,\gamma \infty,k/2)^m \overline{L(f,\gamma \infty,k/2)}^n }{(c_\gamma)^{2s}},  \end{align}
where $\gamma \infty=a_\gamma/c_\gamma$ and $a_\gamma,c_\gamma$ denote the upper-left and lower-left entry of $\gamma$, respectively. Note that $L(f,\gamma \infty,k/2)$ and $c_\gamma$ are indeed well-defined on the double coset $\Gamma_\infty\backslash \Gamma_0(q)/\Gamma_\infty$ and that $\gamma \mapsto \gamma \infty$ defines a bijection 
$$\Gamma_\infty\backslash \Gamma_0(q)/\Gamma_\infty\rightarrow \left\{  a/c \in \Q \mid 0<a<c,\, (a,c)=1,\, q\mid c  \right\}\cup \{\infty \}.$$ 

We will derive the analytic properties of $D^{m,n}(f, s)$ by studying the following generalized {\it Goldfeld Eisenstein series}; 
\begin{align} \label{eis} E^{m,n}(z, s):= \sum_{\gamma\in\Gamma_\infty \backslash \Gamma_0(q)}  L(f,\gamma \infty,k/2)^m \overline{L(f,\gamma \infty,k/2)}^n \Im (\gamma z)^s,  \end{align}
where $\Gamma_\infty= \{\begin{psmallmatrix}1 & n \\ 0 & 1\end{psmallmatrix}\mid n\in \Z\}$ is the stabilizer of $\infty$ in $\Gamma_0(q)$. The series (\ref{lfunc}) and (\ref{eis}) are connected since the constant term in the Fourier expansion of $E^{m,n}(z, s)$ (at $\infty$) is given by 
$$ \frac{\pi^{1/2}y^{1-s}\Gamma(s-1/2)}{\Gamma(s)}D^{m,n}(f, s). $$
For the proof see Lemma \ref{fourier} below. This will allow us to pass analytic information from $E^{m,n}(z,s)$ to $D^{m,n}(f,s)$ as one does in the Langlands--Shahidi method.

In order to get information about the analytic properties of $E^{m,n}(z,s)$, we will use ideas from an unpublished paper by Chinta and O'Sullivan \cite{ChOS02}; we express $E^{m,n}(z,s)$ as a linear combination of certain Poincaré series $G_{\n,\m,l}(z,s)$, which are weight $l$ automorphic forms. This is known as {\it automorphic completion} and will allow us to employ the spectral theory of automorphic forms.

In particular we will use the analytic properties of the {\it higher weight resolvent operators} to recursively understand the pole order at $s=1$ and the leading Laurent coefficient of $G_{\n,\m,l}(z,s)$. The overall strategy can be illustrated as follows:
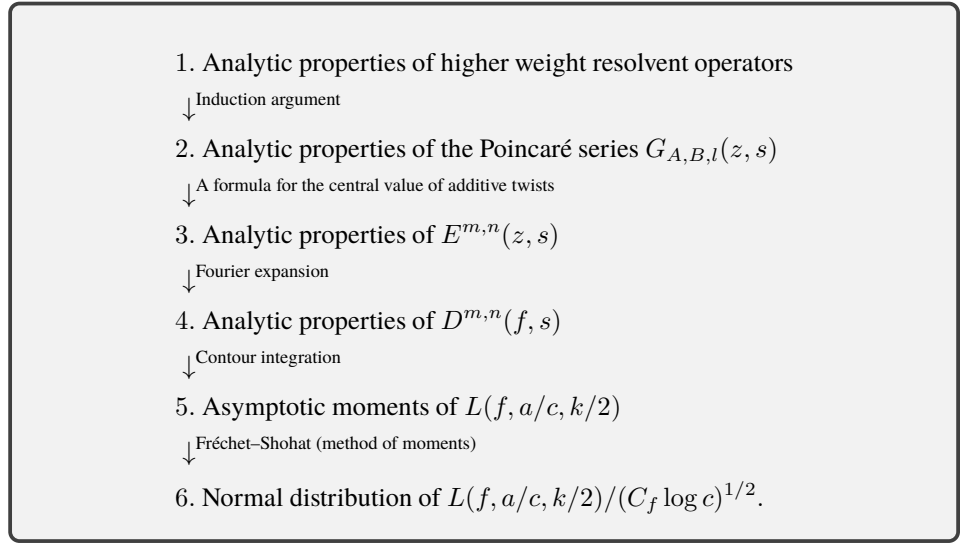
\begin{tcolorbox}
\begin{align*}
&1. \text{ Analytic properties of higher weight resolvent operators} \\
&  \downarrow^{\text{Induction argument}} \\
& 2. \text{ Analytic properties of the Poincaré series }G_{\n,\m,l}(z,s) \\
&\downarrow^{\text{A formula for the central value of additive twists}}\\
& 3. \text{ Analytic properties of }E^{m,n}(z,s)\\
&\downarrow^{\text{Fourier expansion}}\\
&  4. \text{ Analytic properties of }D^{m,n}(f,s) \\
&\downarrow^{\text{Contour integration}}\\
& 5.  \text{ Asymptotic moments of }L(f,a/c,k/2) \\
&\downarrow^{\text{Fréchet--Shohat (method of moments)}}\\
&  6. \text{ Normal distribution of }L(f,a/c,k/2)/(\log c)^{1/2} .
\end{align*}
\end{tcolorbox}

The rest of the paper is structured as follows; in Section \ref{background} we will introduce the needed background on weight $k$ Laplacians and additive twists. In Section \ref{Poincare}, we will study the analytic properties of the Poincaré series $G_{\n,\m,l}(z,s)$. In Section \ref{central}, we will prove the normal distribution of additive twists; in order to keep the exposition as simple as possible, we will restrict to the case of a single cusp form and additive twists corresponding to the cusp $\infty$ and then explain how to extend the methods to the general setting. In Section \ref{centralvalues}, we will present applications to certain \lq wide{\rq} families of automorphic $L$-functions. 
\section{Background: Weight $k$ Laplacians and additive twists}\label{background}
In this section we will recall some standard facts about higher weight Laplacians and additive twists of modular $L$-functions. We will work with a general discrete and co-finite subgroup $\Gamma$ of $\PSL_2(\R)$ with a cusp at $\infty$ of width 1 (see \cite[Section 2]{Iw} for definitions). But one does not lose much by restricting to the case of {\it Hecke congruence subgroups};
$$\Gamma=\Gamma_0(q):=\{\gamma \in \PSL_2(\Z)\text{ s.t. } q\mid c_\gamma\}.$$
\subsection{Weight $k$ Laplacians}
We will refer to \cite[Chapter 4]{DuFrIw02}, \cite[section 2.1.2]{Mi} and \cite[Chapter 4]{Iw} for a more comprehensive account on automorphic Laplace operators.

Let $k$ be an even integer. The space of {\it automorphic functions of weight} $k$ with respect to $\Gamma$ are (measurable) functions $g:\H\rightarrow \C$ satisfying
$$g( \gamma z)=j_\gamma(z)^k g(z),  \quad \text{for all } \gamma\in \Gamma,$$
where $j_\gamma(z):= j(\gamma,z)/|j(\gamma,z)|=(cz+d)/|cz+d|$ with $c,d$ the bottom-row entries of $\gamma$. Note that we have the cocycle relation;
\begin{align}\label{j}j_{\gamma_1\gamma_2}(z)=j_{\gamma_1}(\gamma_2 z) j_{\gamma_2}(z),  \end{align} 
and recall also the useful calculation 
\begin{align}\nonumber \Im (\gamma z)=\frac{\Im z}{|j(\gamma,z)|^2}. \end{align}
Given an automorphic function $g$ of weight $k$, we define the {\it Petersson norm} by
$$  |\!|g|\!|^2 := \int_{\Gamma\backslash \H} |g(z)|^2  d\mu (z), $$
where $d\mu (z)=dxdy/y^2$ is the hyperbolic measure on $\H$. From this we define the {\it Hilbert space of all square integrable weight $k$ automorphic functions};
$$ L^2(\Gamma, k):= \{g\text{ automorphic of weight $k$ with respect to $\Gamma$ s.t. } |\!|g|\!|<\infty\}.$$
(strictly speaking to get a Hilbert space one has to study the quotient modulo the kernel of $|\!|\cdot |\!|$) with inner-product given by
$$ \langle g,h\rangle := \int_{\Gamma\backslash \H} g(z)\overline{h(z)}  d\mu (z), $$
for $g,h\in L^2(\Gamma, k)$. Maa{\ss} defined certain {\it raising-} and {\it lowering operators} acting on $C^\infty(\H)$ (smooth functions on $\H$) as
\begin{align*} K_k:= (z-\overline{z}) \frac{\partial}{\partial z}+\frac{k}{2}, \quad L_k:= (z-\overline{z}) \frac{\partial}{\partial \overline{z}}+\frac{k}{2}.\end{align*}
These operators map between spaces of different weights in the sense that they define maps; 
\begin{align*} &K_k: C^\infty(\H)\cap L^2(\Gamma, k)\rightarrow C^\infty(\H)\cap L^2(\Gamma, k+2), \\
&L_k:C^\infty(\H)\cap L^2(\Gamma, k)\rightarrow C^\infty(\H)\cap L^2(\Gamma, k-2),\end{align*}
for $\Gamma$ as above. 
The raising and lowering operators are adjoint to each other in the following sense;
\begin{align}\label{adjlow}  
\langle  K_k g_k, h_{k+2}   \rangle &=  \langle   g_k, L_{k+2}h_{k+2},   \rangle
\end{align}
and satisfy the following product rule;
\begin{align} 
  \nonumber K_{k+l} (g_k g_l ) &=  (K_k g_k)g_l+ g_k(K_l g_l), \\
  \label{prodlow}   L_{k+l} (h_k h_l ) &=  (L_k h_k)h_l+ h_k(L_l h_l),    
\end{align}
where $g_k, g_l, h_k, h_l$ are smooth automorphic $L^2$-functions of appropriate weights.
\begin{remark}
In most modern expositions the raising operator is denoted $R_k$ (see \cite[Chapter 4]{DuFrIw02}), but in order to avoid confusion with the resolvent operator, we follow the notation of Fay \cite{Fay}. We note that our definition of the lowering operator is equal to minus the one of Fay.\end{remark}

Using the raising and lowering operators, we define the {\it weight $k$ Laplacian} acting on $C^\infty(\H)$ as;
\begin{align}\label{delta} \Delta_k:= K_{k-2}L_k+\lambda(k/2)=L_{k+2}K_k+\lambda(-k/2) , \end{align}
where $\lambda(s)=s(1-s)$. The operator $\Delta_k$ defined on the space of all smooth and rapidly decaying weight $k$ automorphic functions, defines a non-negative and essentially self-adjoint operator on the Hilbert space $L^2(\Gamma, k)$. We denote (by slight abuse of notation) the unique self-adjoint extension also by $ \Delta_k$ with domain $D(\Delta_k)$ dense in $L^2(\Gamma, k)$. We define a {\it Maa{\ss} form} as an eigenfunction $\varphi\in D(\Delta_0)$ of $\Delta_0$.

One sees by direct computation that for $f\in \mathcal{S}_k(\Gamma)$, we have $(z\mapsto y^{k/2}f(z))\in D(\Delta_k)\subset L^2(\Gamma,k)$ and
\begin{align}\label{4.2.8}
L_k y^{k/2}f(z)=0.
\end{align}
Combining this with (\ref{delta}), we see that $y^{k/2}f(z)$ is an eigenfunction for $\Delta_k$ with eigenvalue $\lambda(k/2)=(1-k/2)k/2$.

Using the raising- and lowering operators one can show that 
\begin{align}\label{spectrum} \spec \Delta_k\subset [1/4,\infty) \cup \{\lambda_0=0,\lambda_1,\ldots, \lambda_m\}\cup \{\lambda(1),\ldots, \lambda(k/2)\} \end{align}
where the three sets above correspond to respectively the continuous part constituted by the Eisenstein series, the constant eigenfunction together with the so-called {\it exceptional Maa{\ss} forms} (Maa{\ss} forms with eigenvalue $0< \lambda_i<1/4$) and finally holomorphic cusp forms $\mathcal{S}_j (\Gamma)$ with $2\leq j\leq k$ and $j\equiv k\, (2)$. Recall that by the work of Selberg when $\Gamma=\Gamma_0(q)$ there are always (an abundance of) embedded Maa{\ss} forms (with eigenvalue $\lambda\geq 1/4$).

It is a famous conjecture of Selberg that there are no exceptional Maa{\ss}   forms when $\Gamma=\Gamma_0(q)$ is a Hecke congruence subgroup. Kim and Sarnak \cite{KiSa03} have proved that the smallest eigenvalue $\lambda_1>0$ for a Hecke congruence subgroup satisfies; 
$$  \lambda_1\geq \frac{1}{4}-\left(\frac{7}{64}\right)^2. $$ 
We define the {\it singular set of} $\Gamma$ as;
\begin{align}\label{P}\mathcal{P}_\Gamma:= \{s_0=1, s_1,\ldots, s_m \}, \end{align}
where $ s_i>1/2$ and $\lambda (s_i)=\lambda_i,i=0,\ldots,m$ are the exceptional eigenvalues (together with the trivial  eigenvalue $\lambda=0$). When $\Gamma$ is clear from context, we will shorten notation and write $\mathcal{P}=\mathcal{P}_\Gamma$. The quantity $s_1$ (where we define $s_1=1/2$ if $\mathcal{P}=\{1\}$) will turn out to control the error-terms of our  moment calculation in (\ref{momenter}). Observe that the bound of Kim and Sarnak shows that for Hecke congruence subgroups, we have $\Re s_1\leq 39/64$.

\subsection{Weight $k$ resolvent operators}
Associated to the weight $k$ Laplacian, we have the associated {\it resolvent operator}, which defines a meromorphic operator; 
$$ R(\cdot, k): \{s\in \C\mid \Re s>1/2\} \rightarrow \mathcal{B}(L^2(\Gamma, k)),$$
where $\mathcal{B}(L^2(\Gamma, k))$ denotes the space of bounded operators on the Hilbert space $L^2(\Gamma, k)$. The resolvent operator is (uniquely) characterized by the property:
$$ (\Delta_k-\lambda(s))R(s, k)=\Id_{L^2(\Gamma, k)}, \quad \text{for all }s\in \{s'\in \C\mid \Re s'>1/2\}\backslash (\mathcal{P}\cup\{1,\ldots, k/2\}) . $$
The analytic properties of weight $k$ resolvent operators have been studied intensively by Fay in \cite{Fay}. We will however not use any of these deep results. 

It follows from general properties of resolvent operators and (\ref{spectrum}) that $R(s,k)$ defines a meromorphic operator in the half-plane $\Re s>1/2$ with poles contained in the set $\mathcal{P}\cup \{1,\ldots, k/2\}$ (which is why we called $\mathcal{P}$ the singular set). Furthermore for any $\lambda_0=\lambda(w_0)$ with $\Re w_0>1/2$, we have the following representation in a neighborhood of $w_0$;
\begin{equation}\label{resexp} R(s,k)= \frac{P_{\lambda_0,k}}{s-w_0}+R_{\text{\rm reg}, w_0}(s,k), \end{equation}
where $P_{\lambda_0,k}$ is the projection to the eigenspace of $\Delta_k$ corresponding to the eigenvalue $\lambda_0$ (which might be empty) and $R_{\text{\rm reg},w_0}(s,k)$ is regular at $s=w_0$. 

Finally we also quote the following useful bound on the norm of the resolvent \cite[Appendix A]{Iw}.
\begin{lemma}\label{resbnd}
For $s \in\{s'\in \C\mid \Re s'>1/2\}\backslash (\mathcal{P}\cup \{1,\ldots, k/2\})$, we have 
$$ |\!|R(s,k)|\!|\leq \frac{1}{\dist(\lambda(s), \spec(\Delta_k))},$$
where $|\!|\cdot |\!|$ is the operator norm and $\dist(\cdot, \cdot)$ is the distance function. 
\end{lemma}

\subsection{Additive twists of cuspidal $L$-functions}\label{Additive}
Fix a discrete and co-finite subgroup $\Gamma$ of $\PSL_2(\R)$ with a cusp at infinity of width 1 and let $f\in \mathcal{S}_k(\Gamma)$ be a cusp form of even weight $k$ with Fourier expansion (at $\infty$) given by
$$f(z)=\sum_{n\geq 1}a_f(n) q^n.$$
Then we define the {\it additive twist} (by $r\in \R$) of the $L$-function of $f$;
$$L(f,r,s):=\sum_{n\geq 1} \frac{a_f(n)e(nr)}{n^s},   $$
where $e(z)=e^{2\pi i z}$. We also define $L(f,\infty,s)\equiv 0$.

For all $r\in \R$, the above Dirichlet series converges absolutely for $\Re s>(k+1)/2$ by Hecke's bound \cite[Theorem 3.2]{Iw};
\begin{align}\label{Heckebound} \sum_{n\leq X} |a_f(n)|^2 \ll_f X^k. \end{align}
Furthermore if $r$ corresponds to a cusp (i.e. $r$ is fixed by a parabolic subgroup of $\Gamma$), then $L(f,r,s)$ admits analytic continuation to the entire complex plane (as we will see below).

Associated to additive twists by real numbers of the form $a/c=\gamma \infty $ with $\gamma = \begin{psmallmatrix} a& b \\ c & d \end{psmallmatrix}\in \Gamma,$ we define the completed $L$-function as;
$$  \Lambda(f,a/c,s):=  (2\pi)^{-s}\Gamma(s)c^{s} L(f,a/c,s).  $$
These completed $L$-functions admit analytic continuation, which satisfy the following functional equation (generalizing \cite[Lemma 1.1]{Ju}, see also \cite[Section A.3]{KoMiVa02}). 
\begin{prop} \label{AddFE}
For $\gamma \in \Gamma $, the completed $L$-function $\Lambda(f,a/c, s)$ admits analytic continuation to the entire complex plane, which satisfies the functional equation
$$ \Lambda(f,a/c, s)=(-1)^{k/2}\Lambda(f,-d/c, k-s),  $$
where $a/c=\gamma \infty$ and $-d/c= \gamma^{-1}\infty$. 
\end{prop}
\begin{proof}
We mimic Hecke's proof of analytic continuation and functional equation of cuspidal $L$-functions. In the range of absolute convergence of $L(f,a/c, s)$, we have the following period integral representation;
\begin{align}
\label{periodint}\int_0^\infty f(a/c+iy/c)y^{s}\frac{dy}{y}= \Lambda(f,a/c, s).
\end{align}
If we let $z_\gamma:=-d/c-1/(iyc)$ then we have the following two relations;
$$   \gamma z_\gamma=  \begin{pmatrix} a& b \\ c & d \end{pmatrix}z_\gamma=a/c+iy/c,\quad  j(\gamma, z_\gamma)= -(iy)^{-1}.   $$
This yields 
\begin{align}
\nonumber &\int_0^\infty f(a/c+iy/c)y^{s}\frac{dy}{y}\\
\label{3.2.1}&=\int_0^1 f(a/c+iy/c)y^{s}\frac{dy}{y}+(-1)^{k/2}\int_0^1  f(-d/c+iy/c)y^{k-s}\frac{dy}{y},
\end{align}
using modularity of $f$ and a change of variable $y\mapsto 1/y$. Now we get analytic continuation to the entire complex plane by the vanishing of $f$ at the cusps $a/c=\gamma \infty$ and $-d/c=\gamma^{-1}\infty$. Furthermore (\ref{3.2.1}) yields the functional equation immediately. This completes the proof.
\end{proof}
\begin{remark}\label{DoHoKiLe}
In the special case $\Gamma=\Gamma_0(q)$, the above proposition applies to additive twists by rational numbers $a/c$ where $(a,c)=1$ and $q|c$. The functional equation for twists by general rational numbers is much more involved, see \cite[Theorem 3.1]{DiHoKiLe18}.
\end{remark}
\subsubsection{The convexity bound for additive twists}\label{convexity}
As a basic application of the functional equation, we will derive a preliminary bound for the central value ($s=k/2$) of additive twists by $a/c$, which are $\Gamma$-equivalent to $\infty$, using the Phragmén--Lindelöf principle \cite[Theorem 5.53]{IwKo}. This is known as the {\it convexity bound}.

By the absolute convergence of $L(f,a/c,s)$ for $\Re s>(k+1)/2$, we get for any $\eps >0$
$$ \Lambda(f,a/c,(k+1)/2+\eps+it)\ll_\eps c^{k/2+1/2+\eps} ,  $$
where the implied constant also depends on $f$ (here we also use Stirling's approximation, which shows that $\Gamma(k/2+1/2+\eps +it)$ is bounded in $t$). By the functional equation, we derive similarly that
$$ \Lambda(f,a/c,(k-1)/2-\eps+it )\ll_\eps c^{k/2+1/2+\eps}.   $$
Finally by the period integral representation (\ref{periodint}), we get the bound
$$\Lambda(f,a/c,s)\ll_c 1,$$
for $(k-1)/2-\eps \leq\Re s\leq (k+1)/2+\eps$. Thus the Phragmen-Lindelöf principle applies and we conclude that
\begin{align}\label{prebound} L(f,a/c,k/2)\ll_\eps c^{k/2+1/2+\eps}  c^{-k/2}=c^{1/2+\eps}.\end{align}
Although this is a crude bound, it shows together with \cite[(2.37)]{Iw};
$$ \#\left\{ (a,c)\mid 0\leq a<c, 0<c\leq X ,\begin{pmatrix} a & \ast\\ c & \ast \end{pmatrix}\in \Gamma   \right\} \ll_\Gamma X^2,    $$
that the main generating series $D^{m,n}(f, s)$ (defined in (\ref{lfunc}) for $\Gamma=\Gamma_0(q)$ and in (\ref{lfunc2}) for general $\Gamma$) converges absolutely (and locally uniformly) in some half-plane $\Re s\gg_{m,n} 1$ to an analytic function. It is not hard to show (see Remark \ref{lindelof}) that in fact additive twists satisfy a Lindelöf type bound; $L(f,a/c,k/2)\ll_\eps c^\eps$ for all $\eps>0$.

\subsubsection{Additive twists and the Eichler--Shimura isomorphism} The additively twisted $L$-functions show up in many papers in the analytic theory of $L$-functions in the disguise of the Voronoi summation formula (which is equivalent to the functional equation for additive twists), but they also have arithmetic significance in themselves as they appear in the {\it Eichler--Shimura isomorphism}. We recall how this isomorphism is constructed following \cite[Section 8.2]{Sh94}.
 
Let $f$ be a cusp form of weight $k$ and level $N$ and let $\gamma \in \Gamma_0(N)$. Then we can associate the following $(k-1)$-dimensional real vector 
$$ u_f(\gamma)=\left(\Re \int_{\gamma \infty}^\infty f(z) dz ,\Re \int_{\gamma \infty}^\infty f(z) z dz,\ldots , \Re\int_{\gamma \infty}^\infty f(z) z^{k-2}dz\right). $$
The map $u_f:\Gamma \rightarrow \R^{k-1}$ defines a parabolic co-cycle in group cohomology, i.e. an element of $Z^1_P(\Gamma, X)$ in Shimura's notation where $X=\R^{k-1}$ is a certain $\Gamma$-module.
 
From this we get a map
$$ f\mapsto \{\text{cohomology class of }u_f\}\in H^1_P(\Gamma, X), $$
which by \cite[Theorem 8.4]{Sh94} induces an $\R$-linear isomorphism from $S_k(\Gamma)$ to the parabolic cohomology group $H^1_P(\Gamma, X)$. This is what is known as the Eichler--Shimura isomorphism and it can also be described in terms of the {\it period polynomials}, which were introduced by Eichler \cite{Ei59}; 
$$ \sigma_{f,\gamma}(X):= \frac{1}{(k-1)!}\int_{\gamma\infty}^{\infty} f(z)(z-X)^{k-2}dz. $$
Note that the entries of $u_f(\gamma)$ are the real parts of the coefficients of $\sigma_{f,\gamma}(X)$, up to a scaling by binomial coefficients. The theory of period polynomials has been used to prove important rationality results for (multiplicative twists of) cuspidal $L$-functions \cite{Manin73}.

Now for any $0\leq l \leq k-2$, we have
\begin{align}
\nonumber \int_{\gamma \infty}^\infty f(z) z^l dz &= i \int_0^\infty f(a/c+iy) (a/c+iy)^l dy\\
\label{EiSh}&=\sum_{j=0}^l  \binom{l}{j} (a/c)^{l-j}  \frac{j!}{(-2\pi i)^{j+1}} L(f,a/c, j+1),
 \end{align}
which shows that the special values of additive twists encode the Eichler--Shimura isomorphism. This formula was the starting point for the author in \cite{Nordentoft20.2}, where the distribution of the Eichler--Shimura map is determined.

\section{Poincaré series defined from antiderivatives of cusp forms}\label{Poincare} 
In this section, we will construct a certain Poincaré series $G_{\n,\m,l}(z,s)$ starting from a fixed holomorphic cusp form. Then we will study the analytic properties of these Poincaré series, which will be crucial in proving our main results. The method introduced for studying these Poincaré series might have independent interest. 

Let $\Gamma$ be a co-finite, discrete subgroup of $\PSL_2(\R)$ with a cusp at $\infty$ of width 1 and $f\in\mathcal{S}_k(\Gamma)$ a cusp form of even weight $k$. Then we define for $n\geq 1$;
$$  I_n(z)=I_n(z;f):=\int_{i\infty}^{z}\int_{i\infty}^{z_{n-1}}\cdots \int_{i\infty}^{z_2} \int_{i\infty}^{z_1} f(z_0) dz_0 dz_1 \ldots d_{z_{n-1}}$$
and $I_0(z):=f(z)$. It is clear that we have $I_{n+1}'=I_n$ and thus $I_n$ is the $n$-fold antiderivative of $f$ which vanishes at $\infty$. By taking derivatives, we see that 
\begin{align}\label{7.1.1}I_n(z)=\frac{(-1)^{n-1}}{(n-1)!}\int_{i\infty}^z f(w) (w-z)^{n-1}dw.\end{align}

Furthermore we let $\n,\m$ denote two multi-sets (sets where elements have multiplicities) with all elements contained in $\{0,\ldots, k/2\}$. We call such a multi-set {\it positive} if all elements are positive or if the multi-set is empty. We let $|\n|$ and $|\m|$ denote the sizes of the multi-sets counted with multiplicity. 

For $A,B$ multi-sets of the above type and $l$ an even integer, we define 
\begin{align}\label{G}  G_{\n, \m, l}(z,s): = \sum_{\gamma \in \Gamma_\infty \backslash \Gamma}& j_\gamma (z)^{-l}\left( \prod_{\nn \in \n} \frac{I_\nn(\gamma z)}{(-2i)^{a}}\right) \left( \prod_{\mm \in \m} \frac{\overline{I_\mm(\gamma z)}}{(2i)^{b}}\right) \Im (\gamma z)^{s+\alpha(\n, \m)}, \end{align}
where 
$$\alpha(\n, \m):=\left(\sum_{\nn\in \n}k/2-\nn\right) +\left(\sum_{\mm\in \m}k/2-\mm\right). $$ 
We will see below that these series converges absolutely when $\Re s\gg 1$. We observe that by (\ref{j}) these Poincaré series are (formally) automorphic; 
$$ G_{\n, \m, l}(\gamma z,s)=j_\gamma(z)^l G_{\n, \m, l}(z,s), \quad \gamma\in \Gamma. $$
The scaling $\alpha(\n,\m)$ has the nice property that 
\begin{align} 
\label{pullout1}&G_{\n\cup\{0\},\m,l}(z,s)= y^{k/2}f(z)G_{\n,\m,l-k}(z,s),\\  
\label{pullout}&G_{\n,\m\cup\{0\},l}(z,s)= y^{k/2}\overline{f(z)}G_{\n,\m,l+k}(z,s),
\end{align}
which follows from the modularity of $f$. Observe that with $\n$ and $\m$ as above, we always have $\alpha(\n, \m)\geq 0$, which will be crucial in many argument. We also have the following symmetry
\begin{align}\label{conj}\overline{G_{\n,\m,l}(z,s)}=G_{\m,\n,-l}(z,\overline{s}). \end{align} 
This shows that it is enough to consider the case $l\geq 0$.

Firstly we will show that (\ref{G}) defines an element of $L^2(\Gamma, l)$ in some half-plane following unpublished work of Chinta and O'Sullivan \cite{ChOS02}. 

\begin{lemma} \label{regionofconv}
For $|A|+|B|>0$ the series $G_{\n, \m, l}(z,s)$ converges absolutely (and locally uniformly in $z$ and $s$)  in the half-plane 
$$\Re s> 1+|A|+|B|$$ 
to an element of $L^2(\Gamma, l)$.
\end{lemma}
\begin{proof} By Hecke's bound on the coefficients of cusp forms $|a_f(n)|\ll n^{k/2}$ (coming from (\ref{Heckebound})), we have 
\begin{equation} \label{heckebound}I_n(z) \ll_n \sum_{m=1}^\infty m^{k/2-n} e^{-2\pi my}\ll_{n,k} y^{-(k/2+1-n)},  \end{equation}
using that $r!=\Gamma(r+1)=\int_0^\infty e^{-x}x^rdx$. This gives
\begin{align} 
\nonumber &\sum_{\gamma \in \Gamma_\infty \backslash \Gamma} \Biggr| j_\gamma (z)^{-l}\left( \prod_{\nn \in \n} \frac{I_\nn(\gamma z)}{(-2i)^\nn}\right) \left( \prod_{\mm \in \m} \frac{\overline{I_\mm(\gamma z)}}{(2i)^\mm}\right) \Im (\gamma z)^{s+\alpha(\n, \m)}\Biggr|\\
\label{boundgoldeis}&\ll_{\n,\m}  \sum_{\gamma\in \Gamma_\infty \backslash \Gamma} \Im (\gamma z)^{\sigma - |A|-|B|}\\
\nonumber &= E\left(z, \sigma - |\n|-|\m|\right),\end{align} 
where $s=\sigma+it$. Since the non-holomorphic Eisenstein series converges absolutely for $\Re s>1$, we get that $G_{\n, \m, l}(z,s) $ convergences absolutely (and locally uniformly in $s$ and $z$) in the desired half-plane. 

For any cusp $\mathfrak{b}$ of $\Gamma$ we have by \cite[Lemma 3.2]{PeRi}
$$ \sum_{\Id\neq \gamma \in \Gamma_\infty\backslash \Gamma} (\Im \gamma \sigma_\mathfrak{b} z)^w \ll y^{1-\Re w}, $$
as $y\rightarrow \infty$ (meaning that $\sigma_\mathfrak{b} z\rightarrow \mathfrak{b}$) and uniformly for $\Re w\in[1+\eps, Y]$ with $Y>1+\eps$ and $\eps>0$. Thus we get using (\ref{boundgoldeis}) the following bound
\begin{align*} &G_{\n, \m, l}(\sigma_\mathfrak{b}z,s)\\
&\ll_{\n,\m}  \left( \prod_{\nn \in \n} | I_\nn(\sigma_\mathfrak{b}z)|\right)\left( \prod_{\mm \in \m}|I_\mm(\sigma_\mathfrak{b}z)|\right) \Im (\sigma_\mathfrak{b}z)^{\sigma+\alpha(\n, \m)} +  y^{1-(\sigma-|\n|-|\m|)}, \end{align*}
at any cusp $\mathfrak{b}$ of $\Gamma$, with $\sigma=\Re s\in [1+\eps,Y]$. First of all if $\mathfrak{b}=\infty$, then (\ref{heckebound}) shows that $ I_n(z) \ll e^{-\pi y} $ as $y\rightarrow \infty$, which implies that $G_{\n, \m, l}(z,s)\rightarrow 0$ as $z\rightarrow \infty$ for $\Re s>1+|\n|+|\m|$. At a cusp $\mathfrak{b}\neq \infty$ we see using (\ref{heckebound}) that 
\begin{align*} G_{\n, \m, l}(\sigma_\mathfrak{b}z,s)\ll_{\n,\m}   \Im (\sigma_\mathfrak{b}z)^{-|\n|-|\m|+\sigma} +  y^{1-(\sigma-|\n|-|\m|)}. \end{align*}
Now since $ \Im (\sigma_\mathfrak{b}z)\rightarrow 0$ as $y\rightarrow \infty$, we conclude that $G_{\n, \m, l}(z,s)\rightarrow 0,$ as $z\rightarrow \mathfrak{b}$ for any cusp $\mathfrak{b}$ of $\Gamma$ when $\Re s>1+|\n|+|\m|$. This implies that $G_{\n, \m, l}(z,s)\in L^2(\Gamma, l)$.
 \end{proof}

\subsection{The recursion formula} 
In order to understand the pole structure of $G_{\n,\m,l}(z,s)$, we will use certain recursion formulas involving the resolvent and the raising and lowering operators. First of all we will record how the raising and lowering operators act on the constituents of Poincaré series.
\begin{lemma}
Let $h: \H\rightarrow \C$ be a smooth function and $l$ an even integer. Then we have
\begin{align*}
K_l[h(\gamma z) \Im &(\gamma z)^s j_\gamma(z)^{-l}]\\
&=\left(2i \frac{\partial h}{\partial z}(\gamma z) \Im (\gamma z)^{s+1}+(s+l/2)h(\gamma z) \Im (\gamma z)^s\right)j_\gamma(z)^{-l-2},\\
L_l[h(\gamma z) \Im &(\gamma z)^s j_\gamma(z)^{-l}]\\
&=\left(2i \frac{\partial h}{\partial \overline{z}}(\gamma z) \Im (\gamma z)^{s+1}-(s-l/2)h(\gamma z) \Im (\gamma z)^s\right)j_\gamma(z)^{-l+2}
\end{align*}
for any $\gamma\in \PSL_2(\R)$.
\end{lemma}
\begin{proof}
Using the intertwining relation;
$$ K_l \left(j_\gamma(z)^{-l} F(\gamma z)\right)= j_\gamma(z)^{-l-2} (K_l F)(\gamma z),   $$
valid for any smooth function $F:\H\rightarrow \C$, we reduce the problem to proving the following identity;
$$  K_l h(z)y^s\overset{?}{=} \left(y\left(i\frac{\partial}{\partial x}+ \frac{\partial}{\partial y}\right)+\frac{k}{2}\right)h(z)y^s, $$
and similar for the lowering operator. This can be done by a straightforward calculation. \end{proof}
This yields the following useful formula.
\begin{cor}
Let $h: \H\rightarrow \C$ be a smooth function and $l$ an even integer. Then we have
\begin{align*} (\Delta_l-\lambda(s))[h(\gamma z) \Im (\gamma z)^s j_\gamma(z)^{-l}]=&-4\frac{\partial^2 h }{\partial z\partial \overline{z}}(\gamma z)\Im (\gamma z)^{s+2}j_\gamma(z)^{-l}\\
&-2i(s-l/2)\frac{\partial h}{\partial z}(\gamma z) \Im(\gamma z)^{s+1}j_\gamma(z)^{-l}\\
&+2i(s+l/2)\frac{\partial h}{\partial \overline{z}}(\gamma z) \Im(\gamma z)^{s+1}j_\gamma(z)^{-l} \end{align*}
for $s\in \C$ and $\gamma\in \SL_2(\R)$.
\end{cor}

From the above we will deduce the main recursion formula which will allow us to inductively understand the pole structure of $G_{\n,\m,l}(z,s)$. To write down the formula we will introduce the following convenient notation for $\nn\in \n$;
$$\n_\nn:=(\n\backslash\{\nn\})\cup \{\nn-1\}.$$
In this notation we have for positive multisets $A$ the following useful relation;
$$ \frac{\partial}{\partial z}\prod_{\nn \in \n} I_\nn(z) = \sum_{\nn\in\n} \prod_{\nn' \in \n_\nn} I_{\nn'}(z), $$
by the Leibniz rule. Thus by summing over $\gamma \in \Gamma_\infty\backslash\Gamma$ and using Lemma \ref{regionofconv}, we arrive at the following lemma.
\begin{lemma} \label{induction}
Let $\n,\m$ be positive multisets and $G_{\n,\m,l}(z,s)$ as above. Then we have  
\begin{align}
\label{weightup} K_{l} G_{\n,\m,l}(z,s)=  (s+\alpha(\n,\m)+l/2)G_{\n,\m,l+2}(z,s)-\sum_{\nn\in\n}G_{\n_\nn,\m,l+2}(z,s),\\
\label{weightdown}L_{l} G_{\n,\m,l}(z,s)= -(s+\alpha(\n,\m)-l/2)G_{\n,\m,l-2}(z,s)+\sum_{\mm\in\m}G_{\n,\m_\mm,l-2}(z,s)
\end{align}
and
\begin{align}
\nonumber G_{\n, \m,l}(z,s)= R(s+\alpha(\n,\m),l)\Biggr(&-\sum_{\nn \in \n, \mm \in \m} G_{\n_\nn, \m_\mm,l}(z,s)\\
 \nonumber &+(s+\alpha(\n,\m) -l/2)\sum_{\nn \in \n} G_{\n_\nn, \m,l}(z,s)\\
\label{rec} &+(s+\alpha(\n,\m) +l/2)\sum_{\mm \in \m} G_{\n, \m_\mm,l}(z,s) \Biggr),
\end{align}
valid apriori for $\Re s>1+|\n|+|\m|$.
\end{lemma} 
This lemma will turn out to be extremely useful.
\begin{remark}The recursion formula (\ref{rec}) is the reason why we have $2i$ and $-2i$ in the denominators in the definition of $G_{\n,\m,l}(z,s)$ and why we have the shift $\alpha(\n,\m)$. \end{remark}

Define {\it the total weight} of $\n,\m$ (and of $G_{\n,\m,l}(z,s)$) as the quantity
$$ \Sigma(\n, \m):= \sum_{\nn \in \n}\nn +\sum_{\mm\in\m}\mm.$$
Then we observe that all Poincaré series on the right-hand side in the recursion formula (\ref{rec}) have strictly smaller total weight than the one on the left-hand side. This will allow us to do an inductive argument on the total weight, when determining the pole structure of the Poincaré series. 

As a first application of Lemma \ref{induction}, we will show meromorphic continuation of $G_{\n,\m,l}(z,s)$ to $\Re s>1/2$. Firstly we will handle the case $l=0$ using (\ref{rec}). This case is easiest to handle since the poles of $R(s,0)$ all satisfy $\Re s\leq 1$. Then we will use (\ref{weightup}) and (\ref{weightdown}) to get the result for general (even) weights $l$.

\begin{prop} \label{analyticont}
Let $A,B$ be two multi-sets such that $|A|+|B|>0$ and $l$ an even integer. Then the Poincaré series $G_{\n, \m,l}(z,s)$ admits meromorphic continuation to the half-plane $\Re s>1/2$ satisfying the following;
\begin{enumerate}[(i)]
\item $G_{\n, \m,l}(z,s)$ defines an element of $L^2(\Gamma, l)$ at all regular points.
\item The poles of $G_{\n, \m,l}(z,s)$ in $1/2< \Re s\leq 1$ are contained in $\mathcal{P}$ (defined as in (\ref{P})).
\item $G_{\n, \m,l}(z,s)$ is regular for $\Re s>1$. 
\end{enumerate}  
\end{prop}
\begin{proof}
We prove the claims by an induction on the total weight $\Sigma(\n,\m)$. If $\Sigma(\n,\m)=0$ then by (\ref{pullout1}) and (\ref{pullout}), we can write
\begin{align}\label{00} G_{\n,\m,l}(z,s)= (y^{k/2}f(z))^{|\n|}(y^{k/2}\overline{f(z)})^{|\m|}E_{l-k(|\n|-|\m|)}(z,s). \end{align}
Now $E_{l-k(|\n|-|\m|)}(z,s)$ is meromorphic in $\Re s>1/2$ with poles contained in $\mathcal{P}$ and is regular for $\Re s>1$ (see \cite[Chapter 4]{DuFrIw02}). Furthermore since $f(z)$ decays rapidly at all cusps, the above defines an element of $L^2(\Gamma, l)$ at all regular points. 

Now assume $\Sigma(\n,\m)>0$. By using (\ref{pullout1}) and (\ref{pullout}) we may assume that $\n,\m$ are positive multi-sets. Firstly we consider the case $l=0$. By (\ref{rec}), we can write 
\begin{align*}&G_{\n,\m,0}(z,s)\\
=&R(s+\alpha(\n,\m),0)\left(\text{\it linear combinations of $G_{\n',\m',l}(z,s)$'s with $\Sigma(\n',\m')<\Sigma(\n,\m)$} \right),\end{align*} 
where by the induction hypothesis, all terms inside the parenthesis satisfy the properties $(i),(ii),(iii)$. Since the resolvent operator $R(s+\alpha(\n,\m),0)$ is regular in the half-plane $\Re s>1$ and meromorphic in $\Re s>1/2$ with poles contained in $\mathcal{P}$, the wanted properties follow for $G_{\n,\m,0}(z,s)$ as well. Observe that, if $\alpha(\n,\m)\geq 1$, then the resolvent is actually regular for $\Re s>1/2$. 

Now to get the claim for all positive weights $l$, we do an induction on the weight. For $l\geq 0$, the identity (\ref{weightup}) gives
$$G_{\n,\m,l+2}(z,s)= \frac{K_{l} G_{\n,\m,l}(z,s)+\sum_{\nn\in\n}G_{\n_\nn,\m,l+2}(z,s)}{(s+\alpha(\n,\m)+l/2)}.  $$
We know by the induction hypothesis that all the Poincaré series on the right-hand side of the above satisfy $(i),(ii),(iii)$ of this proposition. So since $s+\alpha(\n,\m)+ l/2$ is non-zero for $\Re s>1/2$, we see that also $G_{\n,\m,l+2}(z,s)$ satisfies $(i),(ii),(iii)$.
 
A similar argument applies to negative weights using (\ref{weightdown}). This finishes the induction and thus the proof. 
\end{proof}

This allows us to extend the range of validity of Lemma \ref{induction} by uniqueness of analytic continuation.
\begin{cor}
The equations (\ref{weightup}), (\ref{weightdown}) and (\ref{rec}) are valid in the half-plane $\Re s>1/2$ as equalities of meromorphic functions.
\end{cor}

\subsection{Bounds on the pole order at $s=1$} 
Next step is to determine the pole order at $s=1$ of $G_{\n,\m,l}(z,s)$. In this section we will prove certain bounds on the pole order. We will proceed by induction relying on the formulas (\ref{weightup}), (\ref{weightdown}) and (\ref{rec}). We firstly need the following key lemma.

\begin{lemma}\label{key} 
Let $\n,\m$ be positive multi-sets and $l$ an even integer. Then we have for $l\geq 0$;
\begin{align}\label{keyeq}
 \langle G_{\n,\m,lk}(z,s), (y^{k/2}f(z))^{l} \rangle
= \frac{ \sum_{\nn\in\n} \langle  G_{\n_\nn,\m,lk}(z,s), (y^{k/2}f(z))^{l}\rangle}{s+\alpha(\n,\m)+ lk/2-1},
\end{align}
and for $-l\leq 0$;
\begin{align}\label{keyeq2}
 \langle G_{\n,\m,-lk}(z,s),(y^{k/2} \overline{f(z)})^{l} \rangle
= \frac{ \sum_{\mm\in\m} \langle  G_{\n,\m_\mm,-lk}(z,s), (y^{k/2}\overline{f(z)})^{l}\rangle}{s+\alpha(\n,\m)+ lk/2-1}.
\end{align}
\end{lemma} 
\begin{proof}
Assume $l\geq 0$ then by the identity (\ref{weightup}), we have
\begin{align*}
& \langle G_{\n,\m,lk}(z,s), (y^{k/2}f(z))^{l} \rangle\\
 = &\frac{\langle K_{kl-2}G_{\n,\m,lk-2}(z,s)+ \sum_{\nn\in\n} G_{\n_\nn,\m,lk}(z,s), (y^{k/2}f(z))^{l} \rangle}{s+\alpha(\n,\m)+lk/2-1}.
\end{align*}
By the adjointness properties of the raising and lowering operators (\ref{adjlow}), we get 
$$ \langle K_{kl-2}G_{\n,\m,lk-2}(z,s), (y^{k/2}f(z))^{l} \rangle=\langle G_{\n,\m,lk-2}(z,s), L_{lk}(y^{k/2}f(z))^{l} \rangle=0, $$
using (\ref{4.2.8}). This yields the desired formula. The case $-l\leq 0$ is proved similarly using (\ref{weightup}).  
\end{proof}

From this we conclude the following key result. 
\begin{prop} \label{polebound}
The pole order of $G_{\n,\m,l}(z,s)$ at $s=1$ is bounded by 
$$  \min (\#\{ \nn\in \n\mid \nn=k/2 \},\#\{ \mm\in \m\mid \mm=k/2 \})+1. $$
\end{prop}
\begin{proof} We will do an induction on the total weight $\Sigma(\n,\m)$. If $\Sigma(\n,\m)=0$ then the result is clear by the properties of the non-holomorphic Eisenstein series. In general by applying modularity (as in (\ref{pullout1}) and (\ref{pullout})), we may assume that both $\n$ and $\m$ are positive. By the symmetry (\ref{conj}) we may also assume that $|\n|\geq |\m|$.

We proceed by induction on the total weight; assume that the total weight is positive; $\Sigma(\n,\m)>0$ and that we have proved the claim for all smaller $\Sigma(\n,\m)$-values. We begin with the case $l=0$. The recursion formula (\ref{rec}) gives the following;
$$ G_{\n,\m,0}(z,s)= R(s+\alpha(\n,\m),0)\left(-\sum_{\nn\in \n, \mm\in \m} G_{\n_\nn, \m_\mm,l}(z,s)+\cdots \right),  $$
where the terms inside the parenthesis satisfy the claim of the proposition by the induction hypothesis. If $\alpha(\n,\m)>0$ then the claim also follows for $G_{\n,\m,0}(z,s)$, since the resolvent operator $R(s+\alpha(\n,\m),0)$ is regular at $s=1$.

If $\alpha(\n,\m)=0$, then we must have
$$  \n=\underbrace{\{k/2,\ldots, k/2\}}_{n},\quad  \m=\underbrace{\{k/2,\ldots, k/2\}}_{m}$$
for some $n\geq m\geq 0$. 

Now we claim that $\langle G_{\n,\m,0}(z,s),1 \rangle$ has a pole of order at most $m+1$. 

To see this we do an induction on $m$. If $m=0$, then by Lemma \ref{key}, we see directly that
$$\langle G_{\n,\m,0}(z,s),1 \rangle=0.$$
If $m>0$ then we get by Lemma \ref{key}
$$ \langle G_{\n,\m,0}(z,s), 1 \rangle = \frac{ m\langle  G_{\n,\m_{k/2},0}(z,s), 1\rangle}{s-1}$$
and by the induction hypothesis, $G_{\n,\m_{k/2},0}(z,s)$ has a pole of order at most $m$, which proves the claim.

We observe that if $G_{\n,\m,0}(z,s)$ has a pole of order greater than $m+1$, then by (\ref{rec}) and the induction hypothesis there has to be an increase in the pole order coming from the pole in the singular expansion of the resolvent (\ref{resexp}). This implies that the leading Laurent coefficient is constant. But we just showed that $\langle G_{\n,\m,0}(z,s), 1 \rangle$ has a pole of order at most $m+1$. This finishes the induction in the case $l=0$.

By using (\ref{weightup}) and (\ref{weightdown}) as in the proof of Proposition \ref{analyticont}, we get by induction the pole bound for all even weights $l$ as well. This finishes the induction and hence the proof.     
\end{proof}

\subsection{Finding the leading pole} 
For $m\neq n$, Proposition \ref{polebound} yields the desired bound needed to prove Theorem \ref{mainthm} (see (\ref{bound}) below). Next step is to determine the exact pole order and leading Laurent coefficient of $G_{\n,\n,0}(z,s)$ at $s=1$ when
$$ \n=\{ \underbrace{k/2,\ldots,k/2 }_{n}\}. $$
By Proposition \ref{polebound} the pole order is bounded by $n+1$ and we will see that this bound is sharp.

\begin{thm} \label{main}
Let 
$$ \n= \{\underbrace{k/2, \ldots, k/2}_n\}. $$
Then $G_{\n,\n,0}(z,s)$ has a pole of order $n+1$ at $s=1$ with leading Laurent coefficient
$$\frac{(n!)^2|\!|f|\!|^{2n}}{((k-1)!)^n \vol(\Gamma)^{n+1}}. $$ 
\end{thm}
\begin{proof}
We do an induction on $n$. For $n=0$ the claim follows by the analytic properties of the non-holomorphic Eisenstein series \cite[(6.33)]{Iw}.

Now assume $n\geq 1$. First of all we see by (\ref{rec}) that 
\begin{align*} &G_{\n,\n,0}(z,s)\\
&= R(s,0) \left( -n^2 G_{\n_{k/2},\n_{k/2},0}(z,s)+nsG_{\n_{k/2},\n,0}(z,s)+nsG_{\n,\n_{k/2},0}(z,s)\right)  \end{align*}
By the bounds on the pole order from Proposition \ref{polebound}, all the terms inside the parentheses above have a pole of order at most $n$. This shows that if $G_{\n,\n,0}(z,s)$ has a pole of order $n+1$, then the leading pole is contained in the image under the projection onto the constant subspace, since (as above) the increase in the pole order has to come from the resolvent.
 
We will show that indeed 
$$\langle G_{\n,\n,0}(z,s),1\rangle/\langle 1,1\rangle, $$
has a pole of order $n+1$ at $s=1$ with the claimed leading Laurent coefficient.

Applying Lemma \ref{key} twice and using the pole bound from Proposition \ref{polebound}, we get; 
\begin{align*}
\langle &G_{\n,\n,0}(z,s),1\rangle\\
&= \frac{\langle nG_{\n,\n_{k/2},0}(z,s),1\rangle}{s-1}\\
&= \frac{\langle n\sum_{\nn \in\n_{k/2}}G_{\n,(\n_{k/2})_\nn,0}(z,s),1\rangle}{(s-1)s}\\
&= \frac{n\langle G_{\n,(\n_{k/2})_{k/2-1},0}(z,s),1\rangle+(\text{\it pole of order at most $n-1$ at $s=1$})}{(s-1)s}, \end{align*}
where 
$$(\n_{k/2})_{k/2-1}=\{ \underbrace{k/2,\ldots,k/2 }_{n-1}, k/2-2\}.$$ 
By repeated applications of Lemma \ref{key} (and Proposition \ref{polebound}), we arrive at
$$\langle G_{\n,\n,0}(z,s),1\rangle= \frac{n\langle G_{\n,\n'\cup\{0\},0}(z,s),1\rangle+(\text{\it pole of order at most $n-1$ at $s=1$})}{(s-1)s\cdots (s+k/2-2)}  $$
where
$$\n'= \{\underbrace{k/2, \ldots, k/2}_{n-1}\}.  $$
Now by applying modularity as in (\ref{pullout}), we get
$$  \langle G_{\n,\n'\cup\{0\},0}(z,s),1\rangle=\langle G_{\n,\n',k}(z,s),y^{k/2}f(z)\rangle.  $$
By a similar repeated application of Lemma \ref{key} (now with $l=k$), we arrive at
\begin{align*}  
&\langle G_{\n,\n,0}(z,s),1\rangle \\
= &\frac{n^2\langle G_{\n',\n',k}(z,s),y^k |f(z)|^2\rangle+(\text{\it pole of order at most $n-1$ at $s=1$})}{(s-1)s\cdots (s+k/2-2)\cdot (s+k/2-1)\cdots (s+k-2)}.
 \end{align*}  
By the induction hypothesis, we know that $G_{\n',\n',0}(z,s)$ has a pole of order $n$ at $s=1$ with leading Laurent coefficient given by 
$$ \frac{((n-1)!)^2|\!|f|\!|^{2n-2}}{((k-1)!)^{n-1} \vol(\Gamma)^{n}}. $$
Thus we see that 
\begin{align*}  \langle G_{\n,\n,0}(z,s),1\rangle/\langle 1,1 \rangle = &\frac{n^2\left\langle \frac{((n-1)!)^2|\!|f|\!|^{2n-2}}{((k-1)!)^{n-1} \vol(\Gamma)^{n}}, y^k |f(z)|^2  \right\rangle }{(k-1)!(s-1)^{n+1}\vol(\Gamma)}\\
 &\qquad\qquad\qquad\qquad+ (\text{\it pole of order at most $n$ at $s=1$}), \end{align*}
which yields the wanted.  \end{proof}

With this theorem established we can improve Proposition \ref{polebound} in the following special case. 

\begin{cor}\label{weightnot0}
Let
$$ \n= \{\underbrace{k/2, \ldots, k/2}_{n}\}$$
and $l\neq 0$ a non-zero even integer. Then the order of the pole of $G_{\n,\n,l}(z,s)$  at $s=1$ is at most $n$.
\end{cor}
\begin{proof} By the symmetry (\ref{conj}), it is enough to prove it for $l>0$. We prove it by induction on $l$. For $l=2$ we get by (\ref{weightdown})
$$G_{\n,\n,2}(z,s) =\frac{K_0G_{\n,\n}(z,s)+nG_{\n_{k/2},\n,2}(z,s)}{s}. $$
From Theorem \ref{main} we know that the leading Laurent coefficient of $G_{\n,\n,0}(z,s)$ is constant, and thus it is annihilated by $K_0$. Furthermore we know by Proposition \ref{polebound} that $G_{\n_{k/2},\n,2}(z,s)$ has a pole of order at most $n$ at $s=1$. Thus we conclude that also $G_{\n,\n,2}(z,s)$ has a pole of order at most $n$ at $s=1$.

Now assume $l>2$. We get again by (\ref{weightdown}) the following;
$$G_{\n,\n,l+2}(z,s)=\frac{K_lG_{\n,\n,l}(z,s)+nG_{\n_{k/2},\n,l+2}(z,s)}{s+l/2}.$$
Thus by the induction assumption and Proposition \ref{polebound}, we see that also $G_{\n,\n,l+2}(z,s)$ has a pole of order at most $n$ at $s=1$. This finishes the induction and hence the proof. 
\end{proof}

\subsection{Growth on vertical lines} 
In this section we will prove bounds on the $L^2$-norm of $G_{\n,\m,l}(z,s)$ with $s$ in a horizontal strip, bounded away from the singular set $\mathcal{P}$. This we will use to get bounds on vertical lines for the main generating series $D^{m,n}(f,s)$ defined in (\ref{lfunc}), which is needed in order to apply a contour integration argument.
 
We will firstly consider the case of total weight zero; $\Sigma(\n,\m)=0$. We will use the idea used in the proof of \cite[Lemma 3.1]{PeRi2}. Following Petridis and Risager, we will in the proof assume that $\Gamma$ has only one cusp for simplicity. The same argument applies in the general case.
\begin{lemma}
Let $\eps>0$ and $s=\sigma+it$ satisfying $1/2+\eps \leq \sigma \leq 3/2$ and $\dist(s, \mathcal{P})\geq \eps$. Let $\n,\m$ be multi-sets such that $|\n|+|\m|>0$ and $\Sigma(\n,\m)=0$ and let $l$ be an even integer. Then we have the following bound;
$$|\!|G_{\n,\m,l}(z,s)|\!| \ll_\eps 1, $$
where the implied constant might depend on $|\n|,|\m|,l$.
\end{lemma}
\begin{proof}
By the assumption $\Sigma(\n,\m)=0$, we can write
$$ G_{\n,\m,l}(z,s)=(y^{k/2}f(z))^{|\n|}(y^{k/2}\overline{f(z)})^{|\m|}E_{l'}(z,s),   $$
with $l'$ appropriately adjusted. 

Let $\mathcal{F}$ be fundamental domain for $\Gamma \backslash \H$ with a cusp at infinity. For $\Re s>1/2$ and $z\in \mathcal{F}$, we write (following Colin de Verdière \cite{CodeVe83});
$$ E_{l'}(z,s)= h(y)y^s+g(z,s), $$
where $g(z,s)\in L^2(\mathcal{F})$ and $h(y)\in C^\infty(0,\infty)$ is smooth with $h(y)=1$ near the cusp at $\infty$. 
Since $E_{l'}(z,s)$ is a formal eigenfunction for the Laplacian, we have
\begin{align*}(\Delta_{l'}-\lambda(s))g(z,s)&=(\Delta_{l'}-\lambda(s))(E_{l'}(z,s)-h(y)y^s)\\
&=\lambda(s)h(y)y^s+ sh'(y)y^{s+1}+h''(y)y^{s+2}-\lambda(s)h(y)y^s\\
&=sh'(y)y^{s+1}+h''(y)y^{s+2}.  \end{align*}
Now we extend $g(z,s)$ periodically to an element of $L^2(\Gamma, l')$. Then the above yields
$$g(z,s)= R(s,l')(sh'(y)y^{s+1}+h''(y)y^{s+2}),$$
i.e. $g(z,s)$ equals the resolvent applied to a function with compact support.
 
Now by the bound on the norm of the resolvent from Lemma \ref{resbnd}, we get 
$$ |\!|g(z,s)|\!|\leq \frac{|\!|sh'(y)y^{s+1}+h''(y)y^{s+2}|\!|}{\dist(\lambda(s), \spec \Delta_{l'})}. $$
Since $\Delta_{l'}$ is self adjoint, all eigenvalues are real. Thus using the assumption $\dist(s, \mathcal{P})\geq \eps$, we get
$$\dist(\lambda(s), \spec \Delta_{l'}) \gg |\Im (\lambda(s))|+\eps= (2\sigma-1)|t|+\eps.$$ 
This gives
$$ |\!|g(z,s)|\!| \ll  \frac{|\!|sh'(y)y^{s+1}+h''(y)y^{s+2}|\!|}{(2\sigma-1)|t|+\eps }\ll \frac{|s|}{|t|+\eps} \ll_\eps 1.$$
Now by the above, we have
\begin{align*}|\!|G_{\n,\m,l}(z,s)|\!|&\\ 
\leq |\!|(y^{k/2}f(z))^{|\n|}&(y^{k/2}\overline{f(z)})^{|\m|} h(y)y^s |\!| + |\!|  (y^{k/2}f(z))^{|\n|}(y^{k/2}\overline{f(z)})^{|\m|} g(z,s)  |\!|. \end{align*}
The second term is bounded by what we showed above and by the rapid decay of $f$, the first term is bounded uniformly in $s$ as well. Thus we conclude $|\!|G_{\n,\m,l}(z,s)|\!|\ll_\eps 1$ as wanted. 
\end{proof}

With this done, we can do the general case by induction on the total weight $\Sigma(\n,\m)$ using the recursion formula (\ref{rec}) and the bound on the operator norm of the resolvent in Lemma \ref{resbnd}.
\begin{prop}\label{Growth}
Let $\eps>0$ and $s=\sigma+it$ satisfying $1/2+\eps \leq \sigma \leq 3/2$ and $\dist(s, \mathcal{P})\geq \eps$. Let $\n,\m$ be multi-sets satisfying $|\n|+|\m|>0$ and $l$ an even integer. Then we have
$$|\!|G_{\n,\m,l}(z,s)|\!| \ll_\eps 1, $$
where the implied constant depends on $|\n|,|\m|, l$.
\end{prop}
\begin{proof}
We proceed by induction. Above we have done the base case so we may assume that $\Sigma(\n,\m)>0$. By applying modularity we may assume that $\n$ and $\m$ are positive (since $y^{k/2}f(z)$ is bounded). Now by (\ref{rec}) and Lemma \ref{resbnd}, we get
\begin{align*}
|\!|G_{\n,\m,l}(z,s)|\!| \leq \frac{|\!|\text{\it RHS of (\ref{rec})}|\!| }{\dist(\lambda(s+\alpha(\n,\m)),\spec(\Delta_{l}))}.  
\end{align*}
By the induction assumption and the triangle inequality, we see that 
$$|\!|\text{\it RHS of (\ref{rec})}|\!| \ll_\eps |t|+1,$$
using 
$$ s+\alpha(\n,\m)\pm l/2\ll |t|+1, $$
where the implied constant depends on $|\n|, |\m|,l $.

Now since the spectrum of $\Delta_l$ is real, we get
$$  \dist(\lambda(s+\alpha(\n,\m)),\spec(\Delta_{l})) \gg_\eps |t(2\sigma-1)|+\eps \gg_\eps |t|+\eps,$$
using the assumption dist($s,\mathcal{P})\geq \eps$. This gives
$$|\!|G_{\n,\m,l}(z,s)|\!| \ll_\eps \frac{|t|+1}{|t|+\eps}\ll_\eps 1, $$
as wanted.  
\end{proof}

\section{Central values of additive twists} \label{central}
In this section we will use the results from the preceding section to study the central values of additive twists. To state our main theorem in the most general version, we will need to work with more general twists than the ones described in the introduction (as was alluded to in Remark \ref{remgen}). To do this we need to introduce some notation. Given a discrete, co-finite subgroup $\Gamma$ of $\PSL_2(\R)$ with a cusp at $\infty$ of width 1 and two cusps $\mathfrak{a}$ and $\mathfrak{b}$ of $\Gamma$ (not necessarily distinct), we define the following set (following \cite{PeRi});
\begin{align}\label{T}T_{\Gamma,\mathfrak{a}\mathfrak{b}}=T_{\mathfrak{a}\mathfrak{b}}:=\left\{ r= a/c \modulo 1\mid \begin{pmatrix} a & b\\ c & d \end{pmatrix}\in \Gamma_\infty \backslash \sigma_\mathfrak{a}^{-1}\Gamma\sigma_\mathfrak{b} / \Gamma_\infty, c>0 \right\},\end{align}
where $\Gamma_\infty$ denotes the parabolic subgroup of $\Gamma$ fixing $\infty$ and $\sigma_\mathfrak{a}$ denotes a (fixed) scaling matrix of $\mathfrak{a}$ (see \cite[(2.1)]{Iw} for background). Observe that $T_{\infty\mathfrak{b}}$ contains exactly the additive twists by the cusps $\Gamma$-equivalent to $\mathfrak{b}$ (thought of as real numbers).
 
Any $r\in T_{\mathfrak{a}\mathfrak{b}}$ uniquely determines an element in the double quotient $\Gamma_\infty \backslash \sigma_\mathfrak{a}^{-1}\Gamma\sigma_\mathfrak{b} / \Gamma_\infty$ \cite[Proposition 2.2]{PeRi}. Thus given $r\in T_{\mathfrak{a}\mathfrak{b}} $, we can define $c(r)$ as the left-lower entry of any such representative. Using this we define 
\begin{align}\label{T(X)}T_{\Gamma,\mathfrak{a}\mathfrak{b}}(X)=T_{\mathfrak{a}\mathfrak{b}}(X):=\{r\in T_{\mathfrak{a}\mathfrak{b}}\mid c(r)\leq X\}.\end{align}
We observe that for $\Gamma=\Gamma_0(q)$, we get $T_{\infty\infty}(X)=T(X)$ with $T(X)$ defined as in (\ref{TQ}). We will below continue to use the shorthand $T(X)=T_{\infty\infty}(X)$, when there is no danger for confusion.  

Using this notation we can now state the most general statement that we can prove with our methods.
\begin{thm} \label{maingeneral}
Let $\Gamma$ be a discrete and co-finite subgroup of $\PSL_2(\R)$ with a cusp at infinity of width 1, $\mathfrak{b}$ a cusp of $\Gamma$, $k$ an even integer and $f_1,\ldots, f_d$ an orthogonal basis for the space of weight $k$ cusps forms $S_k(\Gamma)$ with respect to the Petersson inner product. Then for any fixed box $\Omega\subset \C^d$, we have 
\begin{align*}
&\mathbb{P}_{T_{\infty\mathfrak{b}}(X)}\left( \left(\frac{L(f_i,r,k/2)}{(C_{f_i} \log c(r))^{1/2}}\right)_{1\leq i \leq d} \in \Omega \right)\\
:&= \frac{\#\left\{ r\in T_{\infty\mathfrak{b}}(X) \mid \left(\frac{L(f_i,r,k/2)}{(C_{f_i} \log c(r))^{1/2}}\right)_{1\leq i \leq d} \in \Omega\right\}}{\#T_{\infty\mathfrak{b}}(X)}\\
&=\mathbb{P}\left( (Y_1,\ldots,Y_d)^T\in \Omega\right)+o(1)
\end{align*}
as $X\rightarrow \infty$, where $Y_1,\ldots, Y_d$ are mutually independent random variables all of which are distributed with respect to the standard complex normal distribution $\mathcal{N}_\C(0,1)$ and 
\begin{align}\label{Cf2}C_f=\frac{(4\pi)^{k}|\!|f|\!|^{2}}{(k-1)!\, \vol(\Gamma)},\end{align}
with $|\!|f|\!|$ the Petersson-norm of $f$ and $\vol(\Gamma)$ the hyperbolic volume of $\Gamma\backslash \H$.\\
(Here $\mathbb{P}( (Y_1,\ldots,Y_d)^T\in \Omega)$ denotes the probability of the event $ (Y_1,\ldots,Y_d)^T\in \Omega$.)
\end{thm}
Recall from Section \ref{method} that our strategy of proof is to use the method of moments. To obtain asymptotic formulas for the moments of additive twists, we will be studying the associated generating series. For additive twists at arbitrary cusps this generating series is defined as follows;
\begin{align}\label{lfunc2}D^{m,n}_\mathfrak{b}(f,s):= \sum_{r\in T_{\infty\mathfrak{b}}} \frac{L(f,r,k/2)^m \overline{L(f,r,k/2)}^n}{c(r)^{2s}}.  \end{align}
We study this generating series by studying the associated Goldfeld Eisenstein series (defined in (\ref{generaleisenstein}) below), which is linked to the Poincaré series $G_{\n,\m,l}(z,s)$ via a formula for the central values of additive twists that we will prove shortly. 

To make the proof more readable, we will restrict to the case of $\mathfrak{b}=\infty$ and a single cusp form $f\in \mathcal{S}_k(\Gamma)$. In Section \ref{difcusp} and Section \ref{jointdist}, we will then explain how the proof can be extended to the general case.

\begin{remark}
To make our argument work, we will need to know apriori that (\ref{lfunc2}) converges absolutely in some half-plane $\Re s>\sigma_0$. By the argument given in Section \ref{convexity}, it is enough to show that 
$$L(f,r,k/2)\ll c(r)^K,$$ 
for some $K\geq 0$. Since we do not have a nice functional equation for general additive twists (see the discussion in Remark \ref{DoHoKiLe} above), the easiest way to achieve this seems to be to combine the (generalized) Birch--Stevens formula (see (\ref{birchstevens2}) below) with the convexity bound for the twisted central values $L(f, \chi,1/2)$. We will not go into the details, but just make it clear that one can easily show an apriori polynomial bound of $L(f,r,k/2)$. 
\end{remark}
 
\subsection{A formula for the central value}\label{formulaforcentralvalue} 
In this section we will prove the promised formula, which expresses the generalized Goldfeld series $E^{m,n}(z,s)$ defined in (\ref{eis}) as a sum of the Poincaré series $G_{\n,\m,l}(z,s)$ studied in the preceding section. This generalizes to higher weight what Bruggeman and Diamantis \cite{BrDia16} for weight 2 call {\it automorphic completion}. 

So let $\Gamma$ be a discrete and co-finite subgroup with a cusp at $\infty$ of width 1 and fix a cusp form $f\in \mathcal{S}_k(\Gamma)$ of even weight $k$. Then we will be interested in the central values $L(f,r,k/2)$ of the additive twists by $r\in T=T_{\infty\infty}$, which we will try to relate to the anti-derivatives of $f$ (denoted by $I_n$ above).
 
The starting point is the period integral representation of $L(f,\gamma \infty,s)$ with $\gamma\in \Gamma$ given in (\ref{3.2.1}). A slight variation of this yields with $a/c=\gamma \infty$ the following;
\begin{align*}
i(2\pi)^{-k/2}\Gamma(k/2) L(f,a/c,k/2)=\int_{\gamma \infty}^{i\infty} f(w) \left(\frac{w-a/c}{i}\right)^{(k-2)/2}dw.
\end{align*}
Observe that the integrand above is holomorphic and thus it follows by the vanishing of $f$ at the cusps that we can shift the contour and arrive at   
\begin{align}\nonumber &(-2\pi i)^{-k/2}\Gamma(k/2)L(f,a/c,k/2)\\
\label{1} &= \int_{\gamma \infty}^{\gamma z} f(w) (w-a/c)^{(k-2)/2} dw +\int_{\gamma z}^{i\infty} f(w) (w-a/c)^{(k-2)/2} dw \end{align}
for any $z\in \H$. 

This expression will allow us to prove the following formula for the central value (here it is crucial that $k$ is even).
\begin{lemma} \label{spectrallemma}Let $z\in \H$ and $\gamma=\begin{psmallmatrix}a&b\\c&d\end{psmallmatrix}\in \Gamma $. Then we have
\begin{align}\nonumber L(f,\gamma \infty,k/2)= & \Biggr((-1)^{k/2}\sum_{0\leq j \leq (k-2)/2}  \frac{(\frac{k-2}{2})!}{j!} c^{-j} j(\gamma,z)^{-j} I_{k/2-j}(\gamma z) \\
\label{important}+&\sum_{0\leq j\leq (k-2)/2} (-1)^j\frac{(\frac{k-2}{2})!}{j!} c^{-j}j(\gamma, z)^j I_{k/2-j}(z)\Biggr)\frac{(-2\pi i)^{k/2}} {\Gamma(k/2)},\end{align}
where $I_n$ is the $n$-fold anti-derivative of $f$ defined in (\ref{7.1.1}).
\end{lemma} 
\begin{proof}
We treat the two integrals in (\ref{1}) separately.

Using the fact that
$$ a/c=\gamma \infty =\gamma z+\frac{c^{-1}}{j(\gamma, z)}, $$
we get
\begin{align*}
&\int_{\gamma z}^{i\infty} f(w) (w-a/c)^{(k-2)/2} dw\\ 
&= \int_{\gamma z}^{i\infty} f(w) (w-\gamma z+ \gamma z-a/c)^{(k-2)/2} dw \\
&= \int_{\gamma z}^{i\infty} f(w) \left(w-\gamma z- \frac{c^{-1}}{j(\gamma, z)}\right)^{(k-2)/2} dw\\
&= (-1)^{k/2}\sum_{0\leq j \leq (k-2)/2}  \frac{((k-2)/2)!}{j!} c^{-j} j(\gamma,z)^{-j} I_{k/2-j}(\gamma z)
\end{align*}
using the integral representation (\ref{7.1.1}) of $I_j(z)$. To treat the other integral we use the identity 
$$w-a/c= \gamma\gamma^{-1} w-a/c= -\frac{c^{-1}}{j(\gamma,\gamma^{-1} w)},$$
which yields
\begin{align*} 
\int_{\gamma \infty}^{\gamma z} f(w) (w-a/c)^{(k-2)/2} dw &= \int_{\gamma \infty}^{\gamma z} f(\gamma\gamma^{-1} w) \left(-\frac{c^{-1}}{j(\gamma, \gamma^{-1}w)}\right)^{(k-2)/2} dw\\
&= \int_{i\infty}^{z} f(\gamma w) \left(-\frac{c^{-1}}{j(\gamma, w)}\right)^{(k-2)/2}j(\gamma, w)^{-2} dw
\end{align*}
after the change of variable $w\mapsto \gamma^{-1}w$. Now by using modularity of $f$ and the following identity
$$ \frac{j(\gamma,w)}{c}= w-z+\frac{j(\gamma,z)}{c}, $$
the above equals
\begin{align*}  
&(-1)^{(k-2)/2} \int_{i\infty}^z f(w)j(\gamma,w)^{k-2}\left(-\frac{c^{-1}}{j(\gamma,w)}\right)^{(k-2)/2}dw\\
=&(-1)^{(k-2)/2}  \int_{i\infty}^{z} f(w) \left( w-z+\frac{j(\gamma,z)}{c} \right)^{(k-2)/2} dw\\
= &\sum_{0\leq j\leq (k-2)/2} (-1)^j \frac{((k-2)/2)!}{j!}c^{-j}j(\gamma,z)^j I_{k/2-j}(z).
\end{align*}
\end{proof}
\begin{remark}
Lemma \ref{spectrallemma} is very closely related to the fact that the additive twists $L(f,\cdot, k/2)$ considered as a map on the cusps of $\Gamma$ define a {\it quantum modular form} in the sense of Zagier \cite{Zagier10}, which was a key input for Bettin and Drappeau \cite{BeDr19} in their dynamical proof of the normal distribution of additive twists (in the special case of level 1). In \cite{Nordentoft20} the author, inspired by this connection, proved quantum modularity for central values of additive twists of holomorphic cusp forms of general level and used this to prove a certain \lq reciprocity law{\rq} for multiplicative twists $L(f,\chi,1/2)$. The similarity between \cite[Corollary 8.3]{BeDr19},\cite[Theorem 4.4]{Nordentoft20} and Corollary \ref{spectrallemma} follows by setting $z=r\in \Q$ in (\ref{important}) and using that $I_n(r)=c_n L(f,r,n)$ for some simple constants $c_n$ only depending on $n$ (and $f$). The proofs of \cite[Theorem 4.4]{Nordentoft20} and Corollary \ref{spectrallemma} are also essentially the same. 
\end{remark}
\subsection{Analytic properties of Goldfeld Eisenstein series}
What we would like to do now is the following; take the formula in Lemma \ref{spectrallemma}, sum over $\gamma \in\Gamma_\infty\backslash\Gamma$, use the identity
\begin{align}\label{c} c= \frac{j(\gamma,z)-\overline{j(\gamma,z)}}{2iy} \end{align}
and then finally use the binomial formula to express $E^{m,n}(z,s)$ as a sum of the Poincaré series $G_{\n,\m,l}(z,s)$. The only slight complication is that when $k\geq 4$ we have negative powers of $c$ in our formula for the central values. In order to bypass this we multiply by a power $c^N$ on both sides of (\ref{important}) for some even $N\geq (k-2)/2$. With this in mind, we define the following {\it $N$-shifted Goldfeld Eisenstein series};
\begin{align}\label{N-shift}  E^{m,n}(z,s;N):=\sum_{\gamma\in\Gamma_\infty\backslash \Gamma}c^{N} L(f,\gamma \infty,k/2)^m \overline{L(f,\gamma \infty,k/2)}^n \Im(\gamma z)^s. \end{align}	  
By (\ref{prebound}) we know that the above series converges absolutely (and locally uniformly) for $\Re s\gg 1$. The series $ E^{m,n}(z,s;N)$ also has a very nice Fourier expansion at $\infty$ with constant term related to $D^{m,n}(f,s)$ as we will see below. 

We will now derive the analytic properties of $E^{m,n}(z,s;N)$ from the results of the preceding sections. This is a major step towards our main result. 
\begin{prop} \label{propeis}
Let $N\geq (n+m)(k-2)/2$ be an even integer. Then the Eisenstein series $E^{m,n}(z,s;N)$ admits meromorphic continuation to $\Re s>N/2+1/2$ satisfying the following; 
\begin{enumerate}[(i)]
\item $E^{m,n}(z,s;N)$ is regular for $\Re s> N/2+1$ and all poles in the strip 
$$ N/2+1/2<\Re s\leq   N/2+1$$ 
are contained in the set $\{p+N/2\mid p\in \mathcal{P}\}$.\\ 
\item The pole order of $E^{m,n}(z,s;N)$ at $s=N/2+1$ is bounded by $\min(m,n)+1$.\\
\item $E^{n,n}(z,s;N)$ has a pole at $s=N/2+1$ of order $n+1$ with leading Laurent coefficient
$$(4\pi)^{nk}y^{-N/2}\frac{\binom{N}{N/2}}{2^{N}}  \frac{(n!)^2|\!|f|\!|^{2n}}{((k-1)!)^n \vol(\Gamma)^{n+1}}. $$  
\end{enumerate}
\end{prop} 
\begin{proof}
By (\ref{important}) we can write 
$$ c^{N}L(f,\gamma \infty,k/2)^m \overline{L(f,\gamma \infty,k/2)}^n  $$
as a linear combination of terms of the type
$$h(z)  j(\gamma, z)^{t} \overline{j(\gamma, z)}^{t'} c^{N'} I_{k/2-j_1}(\gamma z) \cdots I_{k/2-j_{m'}}(\gamma z) \overline{I_{k/2-j_{m'+1}}(\gamma z)}\cdots \overline{I_{k/2-j_{m'+n'}}(\gamma z)}    $$
where $h:\H\rightarrow \C$ is a smooth function (this will be a product of $I_j(z)$ for $1\leq j\leq k/2$), $t,t'$ are integers, $m'\leq m$, $n'\leq n$ are non-negative integers and finally $N'$ is a {\it non-negative} integer. By inspecting (\ref{important}), we see that $t$, $t'$ and $N'$ satisfy 
$$t+t'+N'= N-2\sum_{v=1}^{m'+n'}j_v.$$
Now we use (\ref{c}) and expand using the binomial formula (here it is essential that $N'\geq 0$) to get terms of the type
$$ h(z) j(\gamma, z)^{t} \overline{ j(\gamma, z)}^{t'}I_{k/2-j_1}(\gamma z) \cdots I_{k/2-j_{m'}}(\gamma z) \overline{I_{k/2-j_{m'+1}}(\gamma z)}\cdots \overline{I_{k/2-j_{m'+n'}}(\gamma z)}  $$
where now 
\begin{align}\label{t+t'}t+t'= N-2\sum_{v=1}^{m'+n'}j_v.\end{align}
Now we multiply by $\Im(\gamma z)^s$ and use the identity
$$  j(\gamma, z)^{t} \overline{ j(\gamma, z)}^{t'}\Im(\gamma z)^s =y^{(t+t')/2}j_\gamma(z)^{t-t'} \Im (\gamma z)^{s-(t+t')/2}. $$
Thus summing over $\gamma\in \Gamma_\infty \backslash\Gamma$, we can express $E^{m,n}(z,s;N)$ (for $\Re s$ large enough) as a linear combination of terms of the type 
\begin{align}\label{sumbound}h(z) G_{\n,\m,l}(z,s-N/2) \end{align}
where the $h$'s are smooth functions (more precisely; products of powers of $y$'s and $I_j(z)$'s), $|\n|\leq m,|\m|\leq n$ and $l$ is even (which follows from (\ref{t+t'})). Notice that (\ref{t+t'}) fits beautifully with the the factor $\alpha(\n,\m)$ in the definition of $G_{\n,\m,l}(z,s)$, which is why we get the argument $s-N/2$ for all terms. 

Now it follows directly from Proposition \ref{analyticont} that $E^{m,n}(z,s;N)$ has meromorphic continuation to $\Re s>N/2+1/2$ satisfying property $(i)$ of Proposition \ref{propeis}. Furthermore by Proposition \ref{polebound} it follows that the Poincaré series $G_{\n,\m,l}(z,s-N/2)$ has a pole of order at most $\min (m,n)+1$ at $s=N/2+1$. Thus the same is true for $E^{m,n}(z,s;N)$. 

Now finally let us consider the diagonal case $m=n$. We see by Corollary \ref{weightnot0} and Proposition \ref{polebound}, that all terms (\ref{sumbound}) have a pole of order at most $n$, except the one with 
$$\n=\m=\{\underbrace{k/2,\ldots, k/2}_n\}$$
and $l=0$. Now let us calculate the coefficient of $G_{\n,\n,0}(z,s-N/2)$ in the expansion of $E^{n,n}(z,s;N)$; we have
$$  (2\pi )^{nk}c^{N} |I_{k/2}(\gamma z)|^{2n}=\frac{(2\pi )^{nk}}{(2y)^{N}}  |I_{k/2}(\gamma z)|^{2n}  \sum_{v=1}^{N}(-1)^v\binom{N}{v} j(\gamma,z)^v  \overline{j(\gamma,z)}^{N-v}.   $$ 
Now we multiply by $\Im (\gamma z)^s$ and sum over $\gamma \in \Gamma_\infty \backslash \Gamma$. By the pole bound from Corollary \ref{weightnot0}, we see that only the term with $v=N/2$ above can contribute with a pole of order $n+1$ at $s=N/2+1$. Thus we can write
\begin{align*} E^{n,n}(z,s;N)=&(4\pi)^{nk}y^{-N/2}\frac{\binom{N}{N/2}}{2^{N}}G_{\n,\n,0}(z,s-N/2)\\
&+(\text{\it terms with a pole of order at most $n$ at }s=N/2+1),\end{align*}
where
$$\n=\{\underbrace{k/2,\ldots, k/2}_{n}\}.$$
(The extra factor of $2^{nk}$ comes from $2i$ and $-2i$ in the denominator in the definition of $G_{\n,\n,0}(z,s)$). Now the result follows directly from Theorem \ref{main}.   
\end{proof}

\subsection{Analytic properties of $D^{m,n}(f,s)$} 
Using the above we can now extract analytic information about $D^{m,n}(f,s)$ using that it is essentially the constant term in the Fourier expansion of $E^{m,n}(z,s;N)$ at $\infty$.
\begin{lemma}\label{fourier}
Let $N\geq 0$ be an even integer. Then the constant term in the Fourier expansion of $E^{m,n}(z, s;N)$ (at $\infty$) is equal to
$$ \frac{\pi^{1/2}y^{1-s}\Gamma(s-1/2)}{\Gamma(s)}D^{m,n}(f, s-N/2). $$
\end{lemma}
\begin{proof}
By the double coset decomposition (see \cite[Theorem 2.7]{Iw}), we have
$$\Gamma_\infty\backslash \Gamma / \Gamma_\infty \leftrightarrow \left\{ (c,d)\mid 0\leq d<c, \begin{pmatrix}\ast & \ast \\ c & d \end{pmatrix}\in\Gamma \right\}\cup \{(0,1)\}$$ 
Now since $L(f,\gamma \infty,k/2)$ is well-defined in the above double coset and $L(f,\infty,k/2)=0$ per definition, we can write 
 \begin{align*}&E^{m,n}(z, s;N)\\
 &= \sum_{c>0}\sum_{\substack{0\leq d <c }} c^{N}L(f,\gamma_{c,d} \infty,k/2)^m \overline{L(f,\gamma_{c,d} \infty,k/2)}^n\sum_{l\in\Z} \frac{y^s}{|c(z+l)+d|^{2s}}, \end{align*}
where $\gamma_{c,d}$ is any representative of $(c,d)$ in $\Gamma_\infty\backslash \Gamma / \Gamma_\infty$. Now the result follows by computing the inner sum using Poisson summation as in \cite[Section 3.4]{Iw}.  
\end{proof}
With this lemma at our disposal, we can easily derive the analytic properties of $D^{m,n}(f,s)$ from the results already established.

\begin{thm}\label{final} Let $m,n\geq 0$ be non-negative integers. Then the Dirichlet series $D^{m,n}(f,s)$ admits meromorphic continuation to $\Re s>1/2$ satisfying the following. 
\begin{enumerate}[(i)]
\item $D^{m,n}(f,s)$ is regular for $\Re s> 1$ and the poles in the strip 
$$ 1/2< \Re s\leq 1$$ 
are contained in the singular set $\mathcal{P}$.\\ 
\item The pole order of $D^{m,n}(f,s)$ at $s=1$ is bounded by $\min (m,n)+1$.\\
\item $D^{n,n}(f,s)$ has a pole of order $n+1$ at $s=1$ with leading Laurent coefficient  
\begin{align}\label{leadingcoef} \frac{(n!)^2(4\pi)^{nk}|\!|f|\!|^{2n}}{\pi((k-1)!)^n \vol(\Gamma)^{n+1}}.\end{align}
\item For $s=\sigma+it$ with $1/2+\eps \leq \sigma \leq 2$ and $\dist(\lambda(s), \mathcal{P})\geq \eps$, we have the following bound 
\begin{equation}\label{vertbound}D^{m,n}(f,s)\ll_\eps   (1+|t|)^{1/2}, \end{equation}
where the implied constant may depend on $m,n$.
\end{enumerate}
\end{thm} 
\begin{proof}
Fix some even integer $N\geq (n+m)(k-2)/2$. By Lemma \ref{fourier} we have
$$D^{m,n}(f,s-N/2)= \frac{\Gamma(s)}{\pi^{1/2}y^{1-s}\Gamma(s-1/2)} \int_0^1 E^{m,n}(z,s;N) dx.$$
Thus we conclude directly from Proposition \ref{propeis} the following properties; meromorphic continuation of $D^{m,n}(f,s)$, the claim about the position of the possible poles and the bound on the order of the pole at $s=1$. 

Now to treat the case $m=n$, we recall that the leading Laurent coefficient of $E^{n,n}(z,s;N)$ at $s=N/2+1$ is constant. Thus we see directly from Proposition \ref{propeis} that $D^{n,n}(f,s)$ has a pole of order $n+1$ at $s=1$ with leading Laurent coefficient
 $$\frac{\Gamma(N/2+1)}{\pi^{1/2}y^{1-(N/2+1)}\Gamma(N/2+1/2)} (4\pi)^{nk}y^{-N/2}\frac{\binom{N}{N/2}}{2^{N}}  \frac{n!^2|\!|f|\!|^{2n}}{((k-1)!)^n \vol(\Gamma)^{n+1}}.$$
Using that for even $N$, we have
\begin{align*}  \Gamma(N/2+1/2)= \frac{\pi^{1/2}(N-1)\cdots 3\cdot 1}{2^{N/2}} ,\quad \Gamma(N/2+1)=(N/2)!,\\
\frac{\binom{N}{N/2}}{2^{N}}=\frac{N\cdot(N-1)\cdots 1}{2^{N/2} (N/2)! \cdot 2N\cdot 2(N-1)\cdots 2}= \frac{(N-1)\cdots 3\cdot 1}{2^{N/2} (N/2)!},  
\end{align*}
the claim about the leading Laurent coefficient follows.

To get the claim about growth on vertical lines, we need to somehow bring the bounds on the $L^2$-norms of $G_{\n,\m,l}(z,s)$ from Proposition \ref{Growth} into play. 

First step is to integrate $D^{m,n}(f, s-N/2)$ with respect to $y$ over some finite segment, say $[1,2]$, which gives
$$D^{m,n}(f, s-N/2)=  \frac{\Gamma(s)}{\pi^{1/2}\Gamma(s-1/2)} \int_1^2 \int_0^1 y^{s-1}E^{m,n}(z,s;N) dx dy.  $$
By the proof of Proposition \ref{propeis} we can write the above as a linear combination of terms of the type 
$$   \frac{\Gamma(s)}{\Gamma(s-1/2)}  \int_1^2 \int_0^1 h(z) G_{\n,\m,l}(z,s-N/2) dx dy,  $$
with $h(z)$ some smooth function. Since $h(z)$ is bounded in the region $[0,1]\times [1,2]$, the Cauchy-Schwarz inequality implies the following;
\begin{align*}&\int_1^2 \int_0^1y^{1-s} h(z) G_{\n,\m,l}(z,s-N/2) dx dy\\
&\ll  \left(\left(\int_1^2 \int_0^1 y^{4-2\sigma}|h(z)|^2\, dx dy\right)\cdot\left( \int_1^2 \int_0^1 |G_{\n,\m,l}(z,s-N/2)|^2 \frac{dx dy}{y^2}\right)\right)^{1/2}\\
&\ll_h |\!|G_{\n,\m,l}(z,s-N/2)|\!|.    \end{align*}
Thus for $s=\sigma+it$ with 
$$1/2+N/2+\eps \leq \sigma < 3/2+N/2$$ 
and $s-N/2$ bounded at least $\eps$ away from $\mathcal{P}$, we get by Proposition \ref{Growth} that 
$$D^{m,n}(f, s-N/2) \ll_\eps   \frac{|\Gamma(s)|}{|\Gamma(s-1/2)|},$$
where the implied constant may depend on $m,n$ and $f$. 

Now Stirling's formula implies for $s$ in the given range;
$$\frac{\Gamma(s)}{\Gamma(s-1/2)} \ll_{N} (1+|t|)^{1/2},$$
and thus
$$D^{m,n}(f, s)\ll_\eps (1+|t|)^{1/2}$$
for $s=\sigma+it$ with $1/2+\eps \leq \sigma<1$ and $s$ being $\eps$-bounded away from $\mathcal{P}$.  
\end{proof}
\begin{remark}\label{lindelof}
Using Landau's Lemma \cite[Lemma 5.56]{IwKo}, Theorem \ref{final} gives directly the Lindel\"of type bound $  L(f,r, k/2)\ll_\eps c(r)^{\eps} $ in the $c$-aspect. This can however also be proved in a straightforward manner using \cite[Theorem 5.3]{Iw2} and the functional equation for additive twists. 
\end{remark}
\begin{remark}\label{interpolatedvertb}
Clearly, for $\Re s\geq 1+\eps$, we have by absolute convergence the bound $D^{m,n}(f, s)\ll_\eps 1 $, which by the Phragm\'en--Lindel\"of principle implies the improved bound $D^{m,n}(f, \sigma+it)\ll_\eps (1+|t|)^{1-\sigma+\eps} $ for $1/2+\eps \leq \Re s\leq 1+\eps$. 
\end{remark}

\subsection{Normal distribution} 
In this section we will show that the central values 
$$L(f,r,k/2), \quad r\in T=T_{\infty\infty}$$ 
with a suitable normalization are normally distributed when ordered by the size of $c(r)$. This is done by determining all asymptotic moments and then appealing to a 
classical result of Fréchet and Shohat \cite[Theorem B on p. 17]{Serf} known as the {\it methods of moments}.

To evaluate the asymptotic moments, we firstly apply a contour integration argument to the Dirichlet series $D^{m,n}(f,s)$. This allows us to prove the following theorem.
\begin{thm} \label{normal2}
Let $f\in S_k(\Gamma)$ be a cusp form of even weight $k$ and let $m,n$ be non-negative integers. Then we have
\begin{align}  \label{bound}\sum_{\substack{r\in T(X)}} L(f,r,k/2)^m\overline{L(f,r,k/2)}^n \ll X^2\log (X)^{\min(m,n)}.  \end{align}
Let $n$ be a non-negative integer. Then we have 
\begin{align}\label{momentergeneral}\sum_{\substack{r\in T(X)}} |L(f,r,k/2)|^{2n} =P_n(\log X)X^2+O_\eps(X^{\max( 4/3, 2 s_1)+\eps}),  \end{align}
where $s_1\in \mathcal{P}$ corresponds to the smallest positive eigenvalue of $\Delta$ as in (\ref{P}) (here $s_1=1/2$ if $\mathcal{P}=\{1\}$) and $P_n$ is a polynomial of degree $n$ with leading coefficient
\begin{equation}\label{leadingcoef}\frac{2^nn!}{\pi \, \vol(\Gamma)}(C_{f})^n,\end{equation}
with $C_f$ as in (\ref{Cf2}).
\end{thm}




\begin{proof} The bound (\ref{bound}) follows directly from Theorem \ref{final} using Perron's formula as in \cite[Lemma 3.12]{Ti} and a standard contour integration argument.

Similarly, for $m=n$ we apply Perron's formula to $D^{n,n}(f,s)$ as in \cite[Lemma 3.12]{Ti} and shift the contour to the line $\Re s=1/2+\eps$, bounding the contribution from the vertical segments using the bound coming from Remark \ref{interpolatedvertb}. This way we pick up the residues of $D^{n,n}(f,s)X^s/s$ at all the poles of $D^{n,n}(f,s)$ in the strip $1/2+\eps \leq \Re s\leq 1$ (which we know are contained in $\mathcal{P}=\{s_0=1, s_1, \ldots,  s_m\}$) and arrive at
\begin{align*}
&\sum_{r\in T(\sqrt{X})} |L(f,r,k/2)|^{2n}\\
&= \sum_{w\in \mathcal{P}}  \mathrm{res}_{s=w}D^{n,n}(f,s)\frac{X^s}{s}+ \int_{1/2+\eps-iT}^{1/2+\eps+iT} D^{n,n}(f,s)\frac{X^s}{s}ds+O_\eps( X^{1+\eps}T^{-1}),
\end{align*}
for $T>0$ to be chosen appropriately and $X\in 1/2+\Z$ (using also that $L(f,r,k/2)\ll_\eps c(r)^\eps$ coming from Remark \ref{lindelof}). We know from Theorem \ref{final} that the pole of $D^{n,n}(f,s)$ at $s=1$ is of order $n+1$ with leading Laurent coefficient (\ref{leadingcoef}) which implies that 
\begin{align*}
&\mathrm{res}_{s=1}D^{n,n}(f,s)X^s/s\\
&=  \frac{(n!)^2(4\pi)^{nk}|\!|f|\!|^{2n}}{\pi((k-1)!)^n \vol(\Gamma)^{n+1}}\frac{(\log X)^n X}{n!}+a_{n-1}(\log X)^{n-1} X+\ldots +a_0 X , \end{align*}
for certain constants $a_0,\ldots, a_{n-1}$ (depending on $n$ and $f$). Using the bound (\ref{vertbound}) (and doing the substitution $X\mapsto X^2$), we arrive at
$$\sum_{r\in T(X)} |L(f,r,k/2)|^{2n}= P_n(\log X)X^2+ O_\eps(X^{2s_1+\eps}+ X^{2+\eps}T^{-1}+X^{1+\eps}T^{1/2}),$$
with $P_n$ a polynomial of degree $n$ and leading coefficient given by (\ref{leadingcoef}). Choosing $T=X^{2/3}$ we get the wanted.
\end{proof}
\begin{remark}
If we instead considered smooth moments, we would get the improved error-term $O_\eps (X^{2s_1+\eps})$.
\end{remark}
\begin{remark}One can check that in the weight 2 case, the main term agrees with Petridis and Risager \cite[Corollary 7.7]{PeRi}. \end{remark}
From the above we can deduce the asymptotic moments of $L(f,\cdot ,k/2)$ by partial summation.  

\begin{cor} \label{asymmoments}
Let $m\neq n$ be non-negative integers. Then we have 

\begin{equation}\frac{\sum_{\substack{r\in T(X)}}  \left(\frac{L(f,r, k/2)}{(C_f \log c(r))^{1/2}}\right)^m\left(\frac{\overline{L(f,r,k/2)}}{(C_f \log c(r))^{1/2}}\right)^n}{\# T(X)} \rightarrow 0, \end{equation}
as $X\rightarrow \infty$.

Let $n\geq 0$ be a non-negative integer. Then we have
 \begin{equation}\frac{\sum_{\substack{r\in T(X)}}  \left|\frac{L(f,r, k/2)}{(C_f \log c(r))^{1/2}}\right|^{2n}}{\# T(X)} \rightarrow 2^n n!, \end{equation}
 as $X\rightarrow \infty$.
\end{cor}
\begin{proof}
The corollary follows immediately from partial summation using Theorem \ref{asymmoments} and the asymptotic formula $ \# T(X)\sim X^2/(\pi \, \vol (\Gamma))$ coming from \cite[Lemma 3.5]{PeRi}. \end{proof}

Recall that the coordinates $\begin{psmallmatrix}Y \\ Z\end{psmallmatrix}$ of a standard complex normal distribution (or equivalently a standard 2-dimensional normal distribution with diagonal variance-matrix) has moments given by
$$ E(Y^mZ^n)= \begin{cases} (m-1)!!(n-1)!! & \text{if $m$ and $n$ are even}\\ 0& \text{otherwise}  \end{cases}, $$
where $(n-1)!!=(n-1)\cdot (n-3)\cdots 1$. By taking linear combinations of the moments in Corollary \ref{asymmoments}, it follows that the asymptotic moments of
$$\begin{pmatrix}\Re  \frac{L(f,r,k/2)}{(C_f \log c(r))^{1/2}}\\ \Im \frac{L(f,r,k/2)}{(C_f \log c(r))^{1/2}}\end{pmatrix},\quad r\in T(X)$$
as $X\rightarrow \infty$ are the same as those of the 2-dimensional standard normal. This fact and the above corollary allow us to prove Theorem \ref{mainthm}.
\begin{proof}[Proof of Theorem \ref{mainthm}]
We would like to use the result of Fréchet and Shohat coming from probability theory \cite[p. 17]{Serf} mentioned before. To make it fit into the probability theoretical framework of the Fréchet--Shohat Theorem, we consider for each $X>0$ the 2-dimensional random variable  
$$\begin{pmatrix} Y_X(r) \\ Z_X(r)\end{pmatrix}=\begin{pmatrix}\Re  \frac{L(f,r,k/2)}{(C_f \log c(r))^{1/2}}\\ \Im \frac{L(f,r,k/2)}{(C_f \log c(r))^{1/2}}\end{pmatrix},\quad r\in T(X)$$
where the outcome space $T(X)$ is endowed with the discrete $\sigma$-algebra and the uniform probability measure. Note that the Fréchet--Shohat Theorem, as stated in \cite[p. 17]{Serf}, is only directly applicable for 1-dimensional distribution functions, but we can get around this by using the Cramér--Wold Theorem \cite[p. 18]{Serf}, which says that it is enough to check that the moments of all linear combinations of the coordinates (marginal distributions) converge to the expected. To be precise; it follows from Corollary \ref{asymmoments} that for $(a,b)\in \R^2\backslash (0,0)$ the moments of the random variables $aY_X+bZ_X$ converges to the moments of a normal distribution with mean 0 and variance $a^2+b^2$ as $X\rightarrow \infty$. Thus it follows from the Fréchet--Shohat Theorem that the random variables $aY_X+bZ_X$ converges in distribution to the normal distribution with mean 0 and variance $a^2+b^2$ as $X\rightarrow \infty$ (the normal distribution is uniquely determined by its moments).

Now by the Cramér--Wold Theorem it follows that $\begin{psmallmatrix} Y_X \\ Z_X\end{psmallmatrix}$ converges in distribution to the 2-dimensional standard normal distribution as $X\rightarrow \infty$.
\end{proof}

\subsection{Additive twists at a general cusp} \label{difcusp}
We will now explain how to deal with additive twists associated to general cusps $\mathfrak{b}$ and in particular how to prove Theorem \ref{normal}. In the case of weight 2, the modular symbol $\langle \gamma, f\rangle $ is well-defined for $\gamma \in \Gamma_\infty \backslash \Gamma / \Gamma_\mathfrak{b}$ with $\mathfrak{b}$ any cusp of $\Gamma$, which implies that you get a nice Fourier expansion of $E^{m,n}(z,s)$ at every cusp. This is however {\it not} true for additive twists $L(f,\gamma \infty,k/2)$ of $L$-functions of cusp forms $f$ of weight $k\geq 4$. Thus in order to access the cusp $\mathfrak{b}$, we need to consider {\it the generalized Goldfeld series at $\mathfrak{b}$};
\begin{align}\label{generaleisenstein}  E^{m,n,\mathfrak{b}}(z,s):=\sum_{\gamma\in \Gamma_\infty \backslash \Gamma} L(f,\gamma \mathfrak{b},1/2)^m\overline{L(f,\gamma \mathfrak{b},1/2)}^n(\Im \gamma z)^s. \end{align}
The constant term of the Fourier expansion of $E^{m,n,\mathfrak{b}}(z,s)$ at $\mathfrak{b}$ is exactly given by
$$   \frac{\pi^{1/2}y^{1-s}\Gamma(s-1/2)}{\Gamma(s)}D_{\mathfrak{b}}^{m,n}(f, s) $$ 
with $D_{\mathfrak{b}}^{m,n}(f, s)$ defined as in (\ref{lfunc2}). 

Now by a slight modification of Lemma \ref{spectrallemma}, we conclude that for $r\in T_{\infty\mathfrak{b}}$, we have 
\begin{align}
\nonumber &L(f,r,k/2)= \Biggr(\sum_{0\leq j \leq (k-2)/2}  \frac{(\frac{k-2}{2})!}{j!} \left(\frac{z-\mathfrak{b}}{j(\gamma, \mathfrak{b})j(\gamma, z)}\right)^j (-1)^{k/2-j}I_{k/2-j}(\gamma z) \\
\label{speclemma2}&+ \int_\mathfrak{b}^z f(w)\left( \frac{(w-\mathfrak{b})(j(\gamma,z)(w-\overline{z})-\overline{j(\gamma,z)}(w-z))}{2iy j(\gamma, \mathfrak{b})} \right)^{k/2-1} dw\Biggr)\frac{(-2\pi i)^{k/2}} {\Gamma(k/2)}.
\end{align}
Observe that  $j(\gamma, \mathfrak{b})$ is well-defined for $\gamma\in \Gamma_\infty\backslash \Gamma / \Gamma_\mathfrak{b}$. Thus we can define the {\it Goldfeld Eisenstein series at $\mathfrak{b}$};
$$ E^{m,n,\mathfrak{b}}(z,s;N):=\sum_{\gamma\in \Gamma_\infty \backslash \Gamma} j(\gamma, \mathfrak{b})^N L(f,\gamma \mathfrak{b},1/2)^m\overline{L(f,\gamma \mathfrak{b},1/2)}^n(\Im \gamma z)^s, $$ 
whose Fourier expansion at $\mathfrak{b}$ has constant term equal to;
$$\frac{\pi^{1/2}y^{1-s}\Gamma(s-1/2)}{\Gamma(s)}D_{\mathfrak{b}}^{m,n}(f, s-N/2).$$
Similarly to the case $\mathfrak{b}=\infty$ we can write
$$ j(\gamma, \mathfrak{b}) = j(\gamma,z) \frac{\mathfrak{b}-\overline{z}}{2iy}-\overline{j(\gamma,z)} \frac{\mathfrak{b}-z}{2iy} $$
where we consider $\mathfrak{b}$ as a real number. Combining this trick with (\ref{speclemma2}), we see that we can express $E_\mathfrak{b}^{m,n}(z,s;N)$ in terms of the Poincaré series $G_{\n,\m,l}(z,s)$ and we conclude by an argument as in the case $\mathfrak{b}=\infty$ the following.
\begin{thm} \label{generalcusp}
Let $\Gamma$ be a discrete and co-finite subgroup of $\PSL_2(\R)$ with a cusp at $\infty$ of width 1, $\mathfrak{b}$ a cusp of $\Gamma$ and $f\in \mathcal{S}_k(\Gamma)$ a cusp form of even weight $k$. Then we have
$$ \sum_{r\in T_{\infty\mathfrak{b}}(X)}L(f,r,k/2)^m\overline{L(f,r,k/2)}^n \ll X^2 (\log X)^{\min(m,n)}  $$
and 
\begin{align}\label{generalcusp1} \sum_{r\in T_{\infty\mathfrak{b}}(X)}|L(f,r,k/2)|^{2n} =P_n(\log X)X^2+O_\eps(X^{\max( 4/3, 2 s_1)+\eps}), \end{align}
with $P_n$ as in Theorem \ref{normal2} and $s_1$ as in (\ref{P}).\end{thm}
Using the above we deduce easily Theorem \ref{maingeneral} in the case of (marginalizing to) a single cusp form $f\in\mathcal{S}_k(\Gamma)$ with twists at a general cusp $\mathfrak{b}$. The argument being exactly as in the case $\mathfrak{b}=\infty$.

Furthermore, Theorem \ref{generalcusp} allows us to prove Theorem \ref{normal}, which we will need in order to obtain an asymptotic formula for the averages of certain families consisting of automorphic $L$-functions of the form $L(f\otimes \chi, 1/2)$. The proof of this result (Corollary \ref{multiplicativetwistsgeneral}) will be given in Section \ref{centralvalues} below.
\begin{proof}[Proof of Theorem \ref{normal}]In the case where $\Gamma=\Gamma_0(q)$ and $\mathfrak{b}$ corresponds to the real number $0$, we have coming from \cite[p. 47]{Iw} the following scaling matrix;
$$\sigma_0=\sigma_{\mathfrak{b}}  =\begin{pmatrix} 0 & -1/\sqrt{q}\\ \sqrt{q} & 0 \end{pmatrix}. $$ 
This implies that
$$ T_{\infty 0 }=\{r=\frac{a}{c} \modulo 1\mid (a,c)=1, (c,q)=1\}$$ 
and $c(r)=c\sqrt{q}$. 

Thus we conclude from Theorem \ref{generalcusp} with $\Gamma=\Gamma_0(q)$, $f\in\mathcal{S}_k(\Gamma_0(q))$ and $\mathfrak{b}=0$;
\begin{align*}
&\sum_{\substack{0<c\leq X,\\ (c,q)=1}}\sum_{a\in (\Z/c\Z)^\times} |L(f,a/c,k/2)|^{2n}\\
&= \sum_{r\in T_{\infty 0}(\sqrt{q}X)} |L(f,r,k/2)|^{2n}\\
&= \, \frac{q (2C_f)^n \, n!}{\pi\, \vol(\Gamma_0(q)) } (\log X)^n X^2+\sum_{i=0}^{n-1}\beta^d_{f,i} (\log X)^iX^2+O_\eps(X^{\max(4/3, 2s_1)+\eps}).
\end{align*}  
By the approximation towards Selberg's conjecture by Kim and Sarnak \cite[p. 167]{Iw}, we know that $\Re s_1 \leq 39/64<2/3$, which yields exactly Theorem \ref{normal}. \end{proof}

\subsection{The joint distribution of additive twists of a basis of cusp forms}\label{jointdist}
Instead of considering a single cusp form $f$, we can consider an orthogonal basis $f_1,\ldots, f_d$ for $\mathcal{S}_k(\Gamma)$ with respect to the Petersson inner product. We will restrict to the case $\mathfrak{b}=\infty$ and then the discussion in the previous section carries directly over to this setting as well.

In this case for any two sequences $\underline{g}=(g_1,\ldots, g_m), \underline{h}=( h_1,\ldots, h_n)$ with $g_j,h_j\in \{f_1,\ldots, f_d\} $, we define the corresponding Goldfeld Eisenstein series
$$E^{\underline{g},\underline{h}}(z,s):=\sum_{\gamma\in \Gamma_\infty\backslash \Gamma}\left(\prod_{j=1}^m L(g_j,\gamma \infty,k/2)\right)\left(\prod_{j=1}^{n}\overline{L(h_j,\gamma \infty,k/2)}\right)(\Im \gamma z)^s.   $$
The constant term in the Fourier expansion of $E^{\underline{g},\underline{h}}(z,s)$ (at $\infty$) is given by
$$  \frac{\pi^{1/2}y^{1-s}\Gamma(s-1/2)}{\Gamma(s)} D^{\underline{g},\underline{h}}(s), $$
where
$$D^{\underline{g},\underline{h}}(s):=\sum_{r\in T_{\infty\infty}}\frac{\left(\prod_{j=1}^m L(g_j,r,k/2)\right)\left(\prod_{j=1}^{n}\overline{L(h_j,r,k/2)}\right)}{c(r)^{2s}}.$$
We can express the Goldfeld Eisenstein series as a linear combinations of certain Poincaré series, which generalizes $G_{\n,\m,l}(z,s)$ above. Consider tuples of integers $\underline{u}=(u_1,\ldots, u_{m'})$, $\underline{v}=(v_1,\ldots, v_{n'})$ with $m'\leq m, n'\leq n$ and $0\leq u_i,v_j\leq k/2$ and define
 \begin{align} G_{\underline{u}, \underline{v}, l}(z,s):= \sum_{\gamma \in \Gamma_\infty \backslash \Gamma}& j_\gamma (z)^{-l}\left( \prod_{j=1}^{m'} \frac{I_{u_{j}}(\gamma z; g_j)}{(-2i)^{u_{j}}}\right) \left( \prod_{j=1}^{n'} \frac{\overline{I_{v_{j}}(\gamma z; h_j)}}{(2i)^{v_{j}}}\right) \Im (\gamma z)^{s+\alpha(\underline{u},\underline{v})}, \end{align}
where
$$ \alpha(\underline{u},\underline{v})=\left(\sum_{j} k/2-u_{j}\right)+\left(\sum_{j} k/2-v_{j}\right). $$
Then the analogue of Proposition \ref{polebound} holds for the above Poincaré series as well, with essentially the same proof. Using this, it can be shown by the methods from the preceding sections that $D^{\underline{g},\underline{h}}(s)$ has a pole of order at most $\min (m,n)+1$ at $s=1$. Furthermore when $m=n$, we have 
\begin{align*} D^{\underline{g},\underline{h}}(s)&=  \left(\sum_{\sigma,\sigma'\in S_n}\prod_{j=1}^n \langle g_{\sigma(j)},h_{\sigma'(j)} \rangle \right)\frac{(4\pi)^{nk}}{(s-1)^{n+1}\pi((k-1)!)^n \vol(\Gamma)^{n+1}}\\
&+(\text{\it pole order at most $n$ at $s=1$}),\end{align*}  
where $S_n$ denotes the group of permutation on $n$ letters. Observe that this generalizes our previous results since $|S_n|=n!$. In particular $ D^{\underline{g},\underline{h}}(s)$ has a pole of order $n+1$ exactly if $\underline{g}$ and $\underline{h}$ are permutations of each other.

Now consider the $2d$ dimensional real random variable 
\begin{equation}\label{multidimran} (Y_{1,X},Z_{1,X}, \ldots, Y_{d,X}, Z_{d,X})^T\end{equation} 
on the outcome space $T(X)$ endowed with uniform probability measure, defined by 
$$Y_{j,X}(r)=\Re L(f_j,r,k/2)/\sqrt{C_{f_j}\log c(r)},$$
$$Z_{j,X}(r)=\Im L(f_j,r,k/2)/\sqrt{C_{f_j}\log c(r)}$$ 
for $r\in T(X)$ and $j=1,\ldots, d$. Then by the above we can evaluate all asymptotic moments and show using a combination of the results of Fréchet--Shohat and Cramér--Wold that as $X\rightarrow \infty$, this random variables (\ref{multidimran}) converge in distribution $d$ independent standard complex normal distributions.

By combining the methods described in this and the preceding section, one concludes the proof of Theorem \ref{maingeneral}.

\section{Applications to $L(f\otimes \chi,1/2) $}\label{centralvalues} 
We now apply our results to the averages of certain families constructed from the multiplicative twists $L(f\otimes \chi, 1/2)$ and thus giving a proof of Corollary \ref{multiplicativetwistsgeneral}. The connection between multiplicative and additive twists is for primitive characters given by the Birch--Stevens formula \cite[Theorem 2.3]{MaRu19}, but some cleverness has to be applied in order to deal with non-primitive characters.

Our results apply to a newform $f\in \mathcal{S}_k(\Gamma_0(q))$ of even weight $k$ and level $q$ with Fourier expansion (at $\infty$) given by
$$ f(z)=\sum_{n\geq 1} \lambda_f(n)n^{(k-1)/2} q^n, $$
where $ \lambda_f(n)$ denotes the $n$th Hecke eigenvalue of $f$. In what follows it is essential that $f$ is an eigenform for all Hecke operators.

Associated to such a newform $f \in \mathcal{S}_k(\Gamma_0(q))$ and a Dirichlet character $\chi\modulo c$, we define the (naively) twisted $L$-function;
$$L(f,\chi, s):= \sum_{n\geq 1} \frac{\lambda_f(n)\chi(n)}{n^{\frac{k-1}{2}+s}}, $$
which admits analytic continuation and a functional equation. Note that this is {\it not} necessarily equal to (the finite part) of the $L$-function of the automorphic representation $\pi_f\otimes \chi$ (where $\pi_f$ is the automorphic representation corresponding to $f$). However in the special case when $(q,c)=1$, then this is actually true and we will write $L(f\otimes \chi, s)=L(f,\chi, s)$.  
\subsection{Averages of multiplicative twists}
The first step is to establish a connection between additive twists and multiplicative ones. The formula below is a generalization of the Birch--Stevens formula \cite[Theorem 2.3]{MaRu19} to non-primitive characters.  
\begin{prop}[Birch--Stevens formula for non-primitive characters]\label{BirchSteven} \text{ }\\
Let $f\in \mathcal{S}_k(\Gamma_0(q))$ be a newform of weight $k$ and level $q$ and $\chi$ a Dirichlet character $\modulo c$. Then we have
\begin{align}
\nu(f, \chi^*,c/c(\chi)) L(f, \chi^*,1/2)=\sum_{a\in (\Z/c\Z)^\times} \overline{\chi}(a )L(f,a/c,k/2),
\end{align}
and 
\begin{align}
\label{birchstevens2}L(f,a/c,k/2)=\frac{1}{\varphi(c)} \sum_{\chi \modulo c}\chi(a) \nu(f, \chi^*,c/c(\chi)) L(f, \chi^*,1/2),
\end{align}
where $\chi^*\modulo c(\chi)$ denotes the unique primitive character that induces $\chi$ and
$$ \nu(f, \chi, n):= \tau(\overline{\chi})\sum_{\substack{n_1n_2n_3=n,\\ (n_1,q)=1}}\chi(n_1) \mu(n_1)\overline{\chi}(n_2) \mu(n_2) \lambda_f(n_3)n_3^{1/2}.  $$
\end{prop}
\begin{proof}
For $\Re s>1$, we have because of absolute convergence;
\begin{align}\label{cross} \sum_{a\in (\Z/c\Z)^\times} \overline{\chi}(a) L(f,a/c,s+(k-1)/2) = \sum_{n\geq 1} \frac{\lambda_f(n)}{n^s}\left(\sum_{a\in (\Z/c\Z)^\times}\overline{\chi}(a)e(na/c)\right).   \end{align}
The inner sum is a Gauss sum and by \cite[Lemma 3]{Sh75}, we get
$$\sum_{a\in (\Z/c\Z)^\times}\overline{\chi}(a)e(na/c)= \tau(\overline{\chi^*}) \sum_{d\mid (n,c/c(\chi))}d\, \overline{\chi^*}\left(\frac{c}{c(\chi)d}\right) \mu\left(\frac{c}{c(\chi)d}\right) \chi^*\left(\frac{n}{d}\right), $$
where $\chi^*\modulo c(\chi)$ denotes the unique primitive character that induces $\chi$. Plugging this into (\ref{cross}), interchanging the sums and putting $n=dl$, we arrive at
\begin{align*} &\sum_{a\in (\Z/c\Z)^\times} \overline{\chi}(a) L(f,a/c,s+(k-1)/2)\\
&=\tau(\overline{\chi}^*)\sum_{d\mid c/c(\chi)} \overline{\chi^*}\left(\frac{c}{c(\chi)d}\right) \mu\left(\frac{c}{c(\chi)d}\right)d \sum_{l>0} \frac{\lambda_f(dl)}{(dl)^s} \chi^*(l) .  \end{align*}
Now we use that $f$ is a newform, which implies that
$$   \lambda_f(ld)= \sum_{\substack{h\mid (l,d),\\ (h,q)=1}}\mu(h) \lambda_f\left(\frac{l}{h}\right)\lambda_f\left(\frac{d}{h}\right).  $$
With $m=l/h$ and $\delta=d/h$, we get
\begin{align*} &\tau(\overline{\chi^*}) \sum_{\substack{\delta h\mid c/c(\chi)\\ (h,q)=1}}\overline{\chi^*}\left(\frac{c}{c(\chi)\delta h}\right) \mu\left(\frac{c}{c(\chi)\delta h}\right) \frac{\lambda_f(\delta)}{\delta^{s-1}}h^{1-2s}\chi^*(h)\mu(h) \sum_{m>0}\frac{\chi^*(m)\lambda_f(m)}{m^s}\\
  =&\tau(\overline{\chi^*})  L(f\otimes \chi^*,s)\sum_{\substack{\delta h\mid c/c(\chi),\\ (h,q)=1}}\overline{\chi^*}\left(\frac{c}{c(\chi)\delta h}\right) \mu\left(\frac{c}{c(\chi)\delta h}\right) \frac{\lambda_f(\delta)}{\delta^{s-1}}h^{1-2s}\chi^*(h)\mu(h).\end{align*}
Since the sum above is finite, we can extend the equality to $s=1/2$ by analytic continuation. The second equality of this lemma follows from the first by orthogonality of characters.
\end{proof}
\begin{remark} A similar formula has been considered previously by Merel in \cite[Théorème 1]{Merel09}, where the formula is applied in a more algebraic context. \end{remark}



Using this formula, Corollary \ref{multiplicativetwistsgeneral} is an immediate consequence of Theorem \ref{normal}. 

\begin{proof}[Proof of Corollary \ref{multiplicativetwistsgeneral}]
Since $(q,c)=1$, it follows that $(q,c(\chi))=1$ for all Dirichlet characters $\chi$ appearing on the left-hand side of (\ref{2ndmoment}), where $c(\chi)$ denotes the conductor of $\chi$. This implies that we have $L(f,\chi^*,1/2)=L(f\otimes \chi^*,1/2)$.
 
The corollary now follows from Theorem \ref{normal} by expressing the additive twists in terms of $L(f\otimes \chi^*, 1/2)$ using Lemma \ref{BirchSteven}, interchanging the sums, using orthogonality of Dirichlet characters and the fact that
$$ \overline{\nu(f, \chi^*,c/c(\chi)) L(f\otimes \chi^*,1/2)}=\chi(-1) \nu(f, \overline{\chi}^*,c/c(\overline{\chi})) L(f\otimes \overline{\chi}^*,1/2).   $$
\end{proof}
In the special case $n=1$, we derive the following result, which is an average version of the second moment calculation in \cite[Theorem 1.17]{BlFoKoMiMiSa18} with improved error-term. 

\begin{cor}\label{average2ndmoment}
Let $f\in \mathcal{S}_k(\Gamma_0(q))$ be a newform of even weight $k$ and level $q$. Then we have
\begin{align}
\nonumber&\sum_{c\leq X, (q,c)=1 } \frac{1}{\varphi(c)} \sum_{\chi \modulo c} |\nu(f, \chi^*,c/c(\chi))|^2 |L(f\otimes \chi^*,1/2)|^2\\
=  &\frac{q(4\pi)^k|\!|f|\!|^2}{\pi (k-1)!\, \vol(\Gamma_0(q))^2} (\log X) X^2+\beta_{f,1} X^2+O_\eps(X^{4/3+\eps})
\end{align}
with $\chi^*\modulo c(\chi)$ and $ \nu$ as above and $\beta_{f,1}$ a constant.
\end{cor} 
\bibliography{/Users/Soemanden/Documents/MatematikSamlet/Bib-tex/mybib}
\bibliographystyle{amsplain}

\end{document}